\documentclass[reqno,a4paper, 11pt]{amsart}

\usepackage[a4paper=true,pdfpagelabels]{hyperref}
\usepackage{graphicx}
\usepackage{latexsym}
\usepackage{mathabx}
\usepackage{amsmath}
\usepackage{amssymb}
\usepackage[T1]{fontenc}
\usepackage[utf8]{inputenc}
\usepackage{lmodern}
\usepackage{verbatim}

\usepackage{color}   

\newtheorem{theorem}{Theorem}[section]
\newtheorem{lemma}[theorem]{Lemma}
\newtheorem{corollary}[theorem]{Corollary}
\newtheorem{proposition}[theorem]{Proposition}

\newtheorem{lettertheorem}{Theorem}
\newtheorem{letterlemma}[lettertheorem]{Lemma}

\theoremstyle{definition}

\theoremstyle{remark}

\numberwithin{equation}{section}

\newcommand{\set}[1]{\left \{#1\right \}}
\newcommand{\abs}[1]{\left | #1\right |}
\newcommand{\nm}[1]{\left \| #1 \right \|}

\newcommand{\B}{\mathcal{B}}
\newcommand{\D}{\mathbb{D}}
\newcommand{\DD}{\widehat{\mathcal{D}}}
\newcommand{\Dd}{\widecheck{\mathcal{D}}}

\newcommand{\DDD}{\mathcal{D}}

\newcommand{\N}{\mathbb{N}}
\newcommand{\R}{\mathbb{R}}

\newcommand{\C}{\mathbb{C}}

\renewcommand{\phi}{\varphi}

\newcommand{\T}{\mathbb{T}}

\newcommand{\Z}{\mathbb{Z}}

\def\a{\alpha}               
           
     \def\om{\omega}      
       \def\t{\theta}       
                  \def\z{\zeta}

\def\omg{\widehat{\omega}}
\def\nug{\widehat{\nu}}

\renewcommand{\H}{\mathcal{H}}

\newenvironment{Prf}{\noindent{\emph{Proof of}}}
{\hfill$\Box$ }

\newcommand{\Apq}[2]{A^{#1,#2}_\om}

\begin{document}
	\title[ Improvement of a Hardy-Littlewood inequality ]{ Improvement of a Hardy-Littlewood inequality and applications to the boundedness of analytic paraproducts on mixed norm spaces}
	\author[A. moreno]{\'Alvaro Miguel Moreno}
	\address{Departamento de Analisis Matem\'atico, Universidad de M\'alaga, Campus de Teatinos, 
		29071 Malaga, Spain}
	\email{alvarommorenolopez@uma.es}
	
	\author[J. A. Pel\'aez]{Jos\'e \'Angel Pel\'aez}
	\address{Departamento de Analisis Matem\'atico, Universidad de M\'alaga, Campus de Teatinos, 
		29071 Malaga, Spain}
	\email{japelaez@uma.es}
	
	\thanks{This research is supported in part by Ministerio de Ciencia e Innovaci\'on, Spain, project PID2022-136619NB-I00; La Junta de Andaluc{\'i}a, project FQM210.}
	\date{\today}



	\begin{abstract}
        Let $\mathcal{H}(\mathbb{D})$ denote the space of analytic functions in the unit disc $\mathbb{D}=\{z\in\mathbb{C}:|z|<1\}$. For $0<p<\infty$ and $f\in\mathcal{H}(\mathbb{D})$, let be $$M_p^p(r,f)=\int_0^{2\pi}\left |f(re^{i\theta})\right |^p \frac{d\theta}{2\pi} \quad\text{and} \quad M_\infty(r,f) = \sup\limits_{|z|=r}\left |f(z)\right |.$$ 

		For $0<p<q\leq \infty$, Hardy and Littlewood  proved in \cite{HardyLittlewood} the prevalent inequality
        $$
        M_q(r,f)\le C(p,q)\frac{M_p(\rho,f)}{(\rho-r)^{\frac{1}{p}-\frac{1}{q}}}, \quad 0\leq r<\rho\leq 1, \, f\in\mathcal{H}(\mathbb{D}).
        $$
In this paper, 
        we obtain an improvement of this well-known inequality which is employed to characterize the symbols $g\in\mathcal{H}(\mathbb{D})$ such that the analytic paraproducts $$T_gf(z)=\int_0^z f(\zeta)g'(\zeta)d\zeta,\quad S_gf(z)=\int_0^z f'(\zeta)g(\zeta)d\z\;\text{ and }\; M_gf(z)=f(z)g(z),$$ are bounded between two different mixed-norm spaces $$A^{p,q}_\omega=\left \{ g\in \mathcal{H}(\mathbb{D}):\, \int_0^1 M_p^q(r,g) \omega(r)\,dr<\infty\right \}$$  induced by a radial doubling weight $\omega$. En route to the proof of these characterizations, we consider an open Carleson measure problem posed by Luecking in \cite[p. 347]{luecking} and we solve it in a meaningful particular case.
	\end{abstract}
	\maketitle
	\section{Introduction}
	Let $\H(\D)$ denote the space of analytic functions in the unit disc $\D=\{z\in\C:|z|<1\}$ and  $\mathbb{T}=\{z\in\C:|z|=1\}$.
	For $0<p\leq\infty$, the classical Hardy space $H^p$ consists of $f\in\H(\D)$, such that
	$$
	\nm{f}_{H^p}=\sup_{0\leq r <1} M_p(r,f) <\infty,
	$$ 
		where 
	\begin{equation*}\begin{split}
			M_p(r,f) &=\left (\int_0^{2\pi}\abs{f(re^{i\t})}^p \frac{d\theta}{2\pi} \right )^{\frac{1}{p}},\quad 0\leq r < 1,
			\\  M_\infty(r,f) &= \sup\limits_{\abs{z}=r}\abs{f(z)}, \quad 0\leq r < 1.
	\end{split}\end{equation*}
	
	It is well known that for $0<p<q\leq\infty$, there exists a constant $C=C(p,q)>0$ such that
	\begin{equation}\tag{HL}\label{HL}
	M_q(r,f)\leq \frac{C}{(\rho-r)^{\frac{1}{p}-\frac{1}{q}}}M_p(\rho,f),\quad 0\leq r<\rho\leq 1, \, f\in\H(\D),
	\end{equation}
	where $M_p(1,f)=\nm{f}_{H^p}$ and
 we use the convention $\frac{1}{\infty}=0$. This inequality, proved by Hardy and Littlewood in \cite[Theorem 27]{HardyLittlewood} (see also \cite[Theorem 5.9]{Duren}), has served over the last century as a fundamental and widely used tool in numerous rsearch areas of function and operator theory on spaces of analytic functions, see for instance   \cite{Duren,DS,JVA, Zhu} and the references therein. 
 It is worth mentioning that the refinement of classical inequalities in spaces of analytic functions on $\D$ has become an area of intense research activity in recent years. In particular, there is a significant line of research devoted either to finding the best possible constants \cite{DJV,LMN, Perala} or to establishing contractive inequalities \cite{BrevigOrtegaSeipZhao,Kulikov,Llinares}, due to their numerous applications, for instance in number theory.
In this paper, however, our goal is not to determine the optimal constant in \eqref{HL}, but rather to obtain a refinement that enables us to describe the action of certain operators on analytic weighted mixed norm spaces. This leads us to improve the  inequality \eqref{HL} by replacing the integral mean $M_p(\rho,f)$ in 
the right-hand side of \eqref{HL} with an $\ell^q$-norm of the $L^p$-norms of $f_\rho(z)=f(\rho z)$ over intervals of equal length forming a partition of $[0,2\pi]$, determined uniquely by the distance between $r$ and $\rho$. To state this precisely, we introduce the following notation.
	For each $N\in\N$ consider the uniformly distributed partition of $[0,2\pi]$
    , $\{I_{N,l}\}_{l=0}^{N-1}$, defined by
	$$
	I_{N,l} = \left [2\pi\frac{l}{N},2\pi\frac{l+1}{N} \right ],\quad l=0,\ldots N-1.
	$$
	For $0<p<\infty$ and a measurable function $f:\overline{\D}\to \C$ and $0\leq r\leq 1$ we write 
	$$
	\left (f_{[p]}(r) \right )_{N,l} = \left (\int_{I_{N,l}} \abs{f(re^{i\t})}^p\frac{d\theta}{2\pi} \right )^{\frac{1}{p}}.
	$$

	If $x\in \R$, we denote by $E(x)$  the integer number such that  $E(x)\le x<E(x)+1$. 
	\begin{theorem}\label{th1: improvement} 
		Let $0<p<q\leq\infty$. Then,  there exists a constant $C=C(p,q)>0$ such that
		\begin{equation}\label{eq: th improve}
		M_q(r,f)\leq \frac{C}{(\rho-r)^{\frac{1}{p}-\frac{1}{q}}} \nm{ \left\{\left (f_{[p]}(\rho) \right )_{N(r,\rho) ,l}\right \}_{l=0}^{N(r,\rho) -1}}_{\ell^q}\quad 0\leq r<\rho\leq 1, \, f\in\H(\D),
		\end{equation}
where $N(r,\rho)=E\left ( \frac{1}{\rho-r}\right )$.
Moreover, for each $N\in\N$
\begin{equation}\label{eq:trivial}
		\nm{ \left\{\left (f_{[p]}(1) \right )_{N,l} \right\}_{l=0}^{N-1}}_{\ell^q} \leq \nm{f}_{H^p}, \quad f\in H^p.
		\end{equation}
		However, 
	\begin{equation}\label{eq:improve1}
		\sup\limits_{f\in H^p,N\in \N}  \frac{\nm{f}_{H^p}}{	\nm{\left\{\left (f_{[p]}(1) \right )_{N,l}\right\}_{l=0}^{N-1}}_{\ell^q}}=\infty.
		\end{equation}
	\end{theorem}
Here and on the following, we understand that the right hand side of \eqref{eq: th improve} equals infinity in the case $\rho=1$ if $f$ does not have finite boundary values a. e. on $\T$. The mixed norm term appearing in the right-hand side of \eqref{eq: th improve} is reminiscent of the concept of amalgam space, 
which plays an important role in the proof of \eqref{eq: th improve}. In fact, the proof of \eqref{eq: th improve} for $0<q<\infty$ is split into two steps. First, the integral mean $M_q(r,f)$ is bounded above by an amalgam-type expression involving the radial maximal function of $f_\rho$. Second, the boundedness of the classical Hardy-Littlewood maximal operator on amalgam spaces \cite{AguilOrtega,Carton-Lebrun} is used to control the aforementioned expression by the right-hand side of \eqref{eq: th improve}.

As for the proof of \eqref{eq:improve1}, for each $N\in\N$, it is natural to look for a  function $F_N\in H^p$ such that 
$$ \left ((F_N)_{[p]}(1) \right )_{N,l}\asymp \frac{1}{N^{1/p}} \| F_N\|_{H^p}, \quad \text{for each $ l=0,\ldots N-1$.}$$
This approach lead us to the function theory of $H^p$ spaces, and in particular to useful  properties of the boundary values of lacunary series,  which are used  to prove   \eqref{eq:improve1}
     
	In the second part of this work,  we apply  Theorem~\ref{th1: improvement} to  characterize the symbols $g$ such that the analytic paraproducts  are bounded between analytic  mixed norm spaces induced by radial doubling weights. 
More precisely, the extension $\om(z)=\om(|z|)$ of a non-negative function $\om\in L^1([0,1))$, is called a radial weight. For $0<p\leq \infty$ and $0<q<\infty$, the 
 weighted mixed-norm 
 spaces $L^{p,q}_\om$  consists of measurable functions 
 such that
	$$
	\nm{f}_{L^{p,q}_\om}^q =\int_0^1 M_p^q(r,f) \om(r)\,dr.
	$$
We also consider the analytic
 weighted mixed-norm 
 spaces   $A^{p,q}_\om=\H(\D)\cap L^{p,q}_\om$. 
 Throughout this paper, we assume that $\omg(z)=\int_{\abs{z}}^1\om(r)dr>0$ for all $z\in\D$,  otherwise $\Apq{p}{q}=\H(\D)$.
	If  $0<p=q<\infty$, we denote 
$\Apq{p}{p} = A^p_\om$ which is nothing but the 
Bergman space induced by $\om$.
	Analytic mixed norm spaces induced by standard weights first appeared  in the aforementioned  Hardy and Littlewood’s paper  \cite{HardyLittlewood}, although the spaces themselves were not explicitly defined until Flett’s works \cite{F1,F2}. Since then, they have been extensively studied by many authors, for example, such spaces arise naturally in the study of coefficient multipliers on Hardy and weighted spaces \cite{JVA}, as well as in the analysis of  generalized Hilbert operators on weighted Bergman spaces~\cite{GaGiPeSis, PelRathg, PelSeco}. On the other hand, given $g\in \H(\D)$
the
 integral  operators
 $$ T_g(f)(z)=\int_0^z f(\z)g'(\z)\, \, d\z, \quad S_gf(z)=\int_0^z f'(\z)g(\z)d\z,\quad  g\in \H(\D),$$
and the multiplication operator $M_g(f)=gf$ are called  analytic paraproducts due to formula
$$
M_gf(z)= T_gf(z)+S_gf(z) + f(0)g(0).
$$
The boundedness  of analytic paraproducts has been studied on many spaces of analytic functions since the seminal papers  \cite{AC,AS,ASBergman,Duren:Romberg:Shields}, where the authors described  
their action on classical Hardy spaces  and standard Bergman spaces. They are primordial operators within the concrete operator theory on spaces of analytic function whose study is connected 
with many areas of research such as univalent functions \cite{Pom}, the theory of Muchenhoupt and Bekoll\'e-Bonami weights \cite{Aleman:Constantin,AlPe} or
H\"ormander maximal functions  and Carleson measures \cite[Chapters~2 and 4]{PelRat}.        
 In particular, Hu proved in \cite{Hu}  for a certain class of normal weights,  that $T_g:\Apq{p}{q}\to \Apq{p}{q}$ is bounded if and only if $g$ belongs to the classical Bloch space  $\B=\set{f\in\H(\D): \sup_{z\in\D} \abs{f'(z)}(1-\abs{z})<\infty}$.
See also \cite{ZXH} for further results. 
However, as far as we know,  a description of the symbols $g$ such that $L_g:\Apq{p}{q}\to \Apq{s}{t}$ is bounded, where $L_g\in \{ M_g, S_g, T_g\}$ and $0<p,q,s,t<\infty$,  is unknown even for  $\om\equiv 1$.
 Our next results fills this gap in the theory. In order to state them some notation is needed. A radial weight $\om$ belongs to $\DD$ if there exists $C=C(\om)>0$ such that 
	$$
	\omg(r)\le C\omg\left(\frac{1+r}{2}\right),\quad 0\le r<1.
	$$
	We say that a radial weight $\om\in\Dd$ if there exists $K=K(\om)>1$ and $C=C(\om)>1$ such that 
	\begin{equation}\label{eq: def Dcheck}
		\omg(r)\ge C\omg\left(1-\frac{1-r}{K}\right),\quad 0\le r<1.
	\end{equation}
	We denote $\mathcal{D}=\DD\cap\Dd$ for short, and we simply say that $\omega$ is a doubling weight if $\om\in\DDD$. The class $\DDD$ arises  naturally in many topics in operator theory on spaces of analytic functions. In fact, the Bergman projection $P_\om$, induced by a radial weight $\omega$, acts as a bounded and onto operator from $L^\infty$ to $\B$ 
if and only if $\om \in \DDD$  \cite[Theorem~3]{advances}. Moreover,  the Littlewood-Paley formula
\begin{equation}\label{eq: L-P estimate}
		\nm{f}_{\Apq{p}{q}}^q \asymp \int_0^1 M_p^q(r,f')(1-r)^{q}\om(r)dr + \abs{f(0)}^q,\quad f\in\H(\D),
\end{equation}
holds for any  $0<p,q<\infty$ if and only if $\om\in\DDD$. The case $q=p$ was proved  
in \cite[ Theorem 6]{advances} and its proof can be mimicked to prove it for any $0<p\le \infty, 0<q<\infty$.

Prior to stating the results on the boundedness of the analytic paraproducts between analytic weighted mixed norm spaces, let us observe that by  \eqref{eq: L-P estimate}  the boundedness of
 $L_g:\Apq{p}{q}\to \Apq{s}{t}$,  $L_g\in \{ M_g, S_g, T_g\}$,   is equivalent to the following inequality 
	\begin{equation}\label{eq: carleson-type inequalityTg}
		\left (\int_0^1 \left (\int_0^{2\pi}\abs{f^{(n)}(re^{i\t})}^s \abs{G(re^{i\t})}^s\frac{d\t}{2\pi} \right)^{\frac{t}{s}}(1-r)^{m t}\om(r)  dr \right)^\frac{1}{t} \leq C \nm{f}_{\Apq{p}{q}},
	\end{equation}
where $n,m\in \{0,1\}$ and $G\in \{g,g'\}$ in each case. This leads us to the study of the Carleson-type inequality
    \begin{equation}\label{eq: carleson-type inequality}
		\left (\int_0^1 \left (\int_0^{2\pi}\abs{f^{(n)}(re^{i\t})}^s d\mu_r(\t) \right)^{\frac{t}{s}}d\nu(r) \right)^\frac{1}{t} \leq C \nm{f}_{\Apq{p}{q}},
	\end{equation}
for	$n\in\N\cup \{0\}$, arbitrary positive Borel measures $\{\mu_r\}_{r\in [0,1)}$ on $[0,2\pi]$ and any positive Borel measure $\nu$ on $[0,1)$. Luecking in \cite[p. 347]{luecking} already considered \eqref{eq: carleson-type inequality} for $\om=1$, and
states (without providing a detailed proof) that in this particular case,  
\eqref{eq: carleson-type inequality}
is equivalent to the following discrete inequality
\begin{equation}\label{eq:Luecking-discrete}
	\left (\sum_{j=1}^\infty \int_{1-2^{1-j}}^{1-2^{-j}}2^{jt\left (n+ \frac{1}{p}+\frac{1}{q}\right)} \left (\sum_{l=0}^{2^{j+1}-1} \abs{a_{j,l}}^{s} \mu_r\left (\left [2\pi\frac{l}{2^{j+1}}2\pi\frac{l+1}{2^{j+1}}\right]\right) \right)^{\frac{t}{s}} d\nu(r)\right )^{\frac{1}{t}} \leq C \nm{a_{j,l}}_{\ell^{p,q}}.
	\end{equation}  
Here and on the following,  for $0<p,q\leq\infty$ the space $\ell^{p,q}$ consists of all double-indexed sequences $\{a_{j,l}\}_{j,l}\subset \C$ such that
	$
	\nm{\{a_{j,l}\}}_{\ell^{p,q}} = \nm{\left \{\nm{\{a_{j,l}\}_{l}}_{\ell^p}\right \}_{j}}_{\ell^q} < \infty.
	$
However, to the best of our knowledge, describing the measures $\{\mu_r\}_{r\in [0,1)}$ and $\nu$ such that \eqref{eq: carleson-type inequality} holds is an open problem (even for $\om=1$).

\par In Section~\ref{sec:Carleson}, en route to characterize the boundedness of the analytic paraproducts  between analytic weighted mixed norm spaces induced by radial doubling weights,  we provide a detailed proof of the aforementioned Luecking's assertion for \eqref{eq: carleson-type inequality} and $\om\in\DDD$, see Proposition~\ref{propo: carleson type} below. Furthermore, we obtain a description of the measures such that \eqref{eq: carleson-type inequality} holds in the particular case   $d\mu_r(\t)=\abs{G(re^{i\t})}^s\frac{d\theta}{2\pi}$,  $G\in\H(\D)$, and  $\nu\in\DDD$, see Proposition~\ref{propo: caracterización G nu} below.
From now on, we write
	$$
	p'=\begin{cases}
		\frac{p}{p-1},\quad &\text{if}\quad p>1 \\
		\infty,\quad &\text{if}\quad 0<p\leq 1.
	\end{cases}$$

Our next result describes the symbols $g\in \H(\D)$ for which $T_g$ acts as a bounded operator between two analytic weighted mixed norm spaces.  
    \begin{theorem}\label{th2: charact Tg}
		Let $0<p,q,s,t<\infty$, $g\in\H(\D)$, $\om \in \DDD$, $K=K(\om)\in \N \setminus \{1\}$ such that \eqref{eq: def Dcheck} holds and denote $\frac{1}{\tilde{p}}=\frac{1}{s}-\frac{1}{p}$ and $\frac{1}{\tilde{q}}=\frac{1}{t}-\frac{1}{q}$. Then, the following conditions are equivalent:
		\begin{itemize}
			\item[(i)] $T_g:\Apq{p}{q}\to \Apq{s}{t}$ is bounded;
			\item[(ii)] For $j\in\N\cup\{0\}$, let $r_{j}=1-K^{-j}$. Then,
			$$\left \{K^{j(\frac{1}{p}-1)}\omg(r_j)^{\frac{1}{\tilde{q}}}  \left (g'_{[s]}(r_{j-1}) \right)_{K^{j+2},l} \right \}\in \ell^{s\left (\frac{p}{s}\right )',t\left (\frac{q}{t}\right )'}.$$
			\item[(iii)] The following holds:
			\begin{itemize}
				\item[(a)] If $p\leq s$ and $q\leq t$, 
				$$
				\sup_{z\in\D} \abs{g'(z)}\omg(z)^{\frac{1}{\tilde{q}}}(1-|z|)^{1+\frac{1}{\tilde{p}}}<\infty.
				$$
				\item[(b)] If $s<p$ and $q\leq t$,
				$$
				\sup_{0\le r<1} (1-r)\omg(r)^{\frac{1}{\tilde{q}}}M_{\tilde{p}}(r,g')<\infty.
				$$
				\item[(c)] If $p\leq s$ and $t<q$,
				$$
				|g'(z)|(1-|z|)^{1+\frac{1}{\tilde{p}}}\in L_\om^{\infty, \tilde{q}}.
				$$
				\item[(d)] If $s<p$ and $t<q$,
				$$
				g\in\Apq{\tilde{p}}{\tilde{q}}.
				$$
			\end{itemize}  
		\end{itemize}
		Moreover,
		$$
		\nm{T_g}_{\Apq{p}{q}\to\Apq{s}{t}} \asymp \nm{\left \{K^{j(\frac{1}{p}-1)}\omg(r_j)^{\frac{1}{\tilde{q}}}  \left (g'_{[s]}(r_{j-1}) \right)_{K^{j+2},l} \right \}}_{\ell^{s\left (\frac{p}{s}\right )',t\left (\frac{q}{t}\right )'}}\asymp \rho_{p,q,s,t,\om}(g),
		$$
		where,
		$$
		\rho_{p,q,s,t,\om}(g)= \begin{cases}
			\sup_{z\in\D} \abs{g'(z)}\omg(z)^{\frac{1}{\tilde{q}}}(1-\abs{z})^{1+\frac{1}{\tilde{p}}}, & \text{if}\; p\leq s, q\leq t \\ \sup_{0\le r<1} (1-r)\omg(r)^{\frac{1}{\tilde{q}}}M_{\tilde{p}}(r,g'),
			& \text{if}\;s<p, q\leq t \\ \nm{ (1-|\cdot|)^{1+\frac{1}{\tilde{p}}}\;g'}_{L_\om^{\infty, \tilde{q}}},
			& \text{if}\;p\leq s, t<q \\ \nm{g-g(0)}_{\Apq{\tilde{p}}{\tilde{q}}}
			& \text{if}\;s<p,t<q .
		\end{cases} 
		$$
	\end{theorem}
	Notice that, if $p=s$ then $\frac{1}{\tilde{p}}$ simply denotes the value $0$, and if $s<p$ , then $\tilde{p}=s\left (\frac{p}{s}\right )'>0$, so the quantities and exponents in the conditions of Theorem~\ref{th2: charact Tg} are well-defined.

\par The equivalence between (i) and (ii) in Theorem~\ref{th2: charact Tg} is obtained in Section~\ref{sec:Carleson}  via the solution to the Carleson-type problem \eqref{eq: carleson-type inequality},  using results related to an atomic decomposition of $\Apq{p}{q}$ and a discretization of the $\Apq{p}{q}$ norm previously obtained in \cite{PelRatSierra}.
In Section~\ref{sec:continuous}, Theorem~\ref{th1: improvement} is strongly used to prove the implication (ii)$\Rightarrow$(iii) of Theorem~\ref{th2: charact Tg}.
Finally, the implication (iii)$\Rightarrow$(i) of Theorem~\ref{th2: charact Tg} is obtained using standard ideas.

It is important to highlight that  the equivalence (i)$\Leftrightarrow$(iii) of Theorem~\ref{th2: charact Tg} allows one to prove that under certain  assumptions on the paremeters  $p,q,s,t$ the boundedness of $T_g$ is only possible for constant functions. We summarize these cases in the following corollary.
	\begin{corollary}\label{coro: Tg constant}
		Let $0<p,q,s,t<\infty$, $g\in\H(\D)$,  $\om \in \DDD$ and denote $\frac{1}{\tilde{p}}=\frac{1}{s}-\frac{1}{p}$ and $\frac{1}{\tilde{q}}=\frac{1}{t}-\frac{1}{q}$. If any of the following conditions holds, then $T_g:\Apq{p}{q}\to \Apq{s}{t}$ is bounded if and only if $g$ is a constant function:
		\begin{itemize}
			\item[(a)] If $p\leq s$, $q\leq t$ and $\lim_{r\to 1^-}\omg(r)^{\frac{1}{\tilde{q}}}(1-r)^{1+\frac{1}{\tilde{p}}}=\infty$.
			\item[(b)] If $s<p$, $q\leq t$ and $\lim_{r\to 1^-}\omg(r)^{\frac{1}{\tilde{q}}}(1-r)=\infty$.
			\item[(c)] If $p\leq s$, $t<q$ and $\int_0^1 (1-r)^{ (1+\frac{1}{\tilde{p}})\tilde{q}}\om(r)\,dr=\infty$. 
		\end{itemize}
	\end{corollary}

The proof of Theorem~\ref{th2: charact Tg} can be mimicked, with  appropriate modifications, to obtain a characterization of the boundedness of $S_g$ and $M_g$ between two analytic weighted mixed norm spaces.  We only write here the continuous characterization to lighten the notation, the discrete characterizations are provided in Corollary~\ref{coro: discrete caract Sg Mg} below.

\begin{theorem}\label{th2: charact SgMg}
		Let $0<p,q,s,t<\infty$, $g\in\H(\D)$,  $\om \in \DDD$ and denote $\frac{1}{\tilde{p}}=\frac{1}{s}-\frac{1}{p}$ and $\frac{1}{\tilde{q}}=\frac{1}{t}-\frac{1}{q}$.
Then, the following conditions are equivalent:
		\begin{itemize}
			\item[(i)] $S_g:\Apq{p}{q}\to \Apq{s}{t}$ is bounded;
            \item[(ii)] $M_g:\Apq{p}{q}\to \Apq{s}{t}$ is bounded;
			\item[(iii)] The following holds:
			\begin{itemize}
				\item[(a)] If $p\leq s$ and $q\leq t$, 
				$$
				\sup_{z\in\D} \abs{g(z)}\omg(z)^{\frac{1}{\tilde{q}}}(1-|z|)^{\frac{1}{\tilde{p}}}<\infty.
				$$
                Moreover, if $p=s$ and $q=t$, the condition holds if and only if $g\in H^\infty$. Otherwise the condition holds if and only if $g\equiv 0$.
				\item[(b)] If $s<p$ and $q\leq t$,
				$$
				\sup_{0\le r<1} \omg(r)^{\frac{1}{\tilde{q}}}M_{\tilde{p}}(r,g)<\infty.
                $$
               Moreover, if $q=t$, the condition holds if and only if $g\in H^{\tilde{p}}$. Otherwise the condition holds if and only if $g\equiv 0$.
			      \item[(c)] If $p\leq s$ and $t<q$,
				$$
				|g(z)|(1-|z|)^{\frac{1}{\tilde{p}}}\in L_\om^{\infty, \tilde{q}}.
				$$
                Moreover, if $\int_0^1 (1-r)^{ \frac{\tilde{q}}{\tilde{p}}}\om(r)\,dr=\infty$ the condition holds if and only if $g\equiv 0$.
				\item[(d)] If $s<p$ and $t<q$,
				$$
				g\in\Apq{\tilde{p}}{\tilde{q}}.
				$$
			\end{itemize}  
		\end{itemize}
		Moreover,
		$$
		\nm{S_g}_{\Apq{p}{q}\to\Apq{s}{t}} \asymp \nm{M_g}_{\Apq{p}{q}\to\Apq{s}{t}} \asymp  \begin{cases}
			\sup_{z\in\D} \abs{g(z)}\omg(z)^{\frac{1}{\tilde{q}}}(1-\abs{z})^{\frac{1}{\tilde{p}}}, & \text{if}\; p\leq s, q\leq t \\ \sup_{0\le r<1} \omg(r)^{\frac{1}{\tilde{q}}}M_{\tilde{p}}(r,g),
			& \text{if}\;s<p, q\leq t \\ \nm{ (1-|\cdot|)^{\frac{1}{\tilde{p}}}\;g}_{L_\om^{\infty, \tilde{q}}},
			& \text{if}\;p\leq s, t<q \\ \nm{g}_{\Apq{\tilde{p}}{\tilde{q}}}
			& \text{if}\;s<p,t<q .
		\end{cases} 
		$$
	\end{theorem}
Bearing in mind that for $\beta>0$ and $\om\in\DDD$, $\om_{[\beta]}(r)=\om(r)(1-r)^\beta \in \DDD$ \cite{PelRat,PelRosa} and \eqref{eq: L-P estimate}, it follows that $M_g$, $S_g$, and $T_g$ are simultaneously bounded from $\Apq{p}{q}$ to $\Apq{s}{t}$ whenever $t<q$. However, if $t\ge q$, a cancellation phenomenon may happen, so that $T_g$ is bounded while $M_g$ and $S_g$ are not.

The rest of the paper is organized as follows. In Section~\ref{sec:2} it is provided a proof of Theorem~\ref{th1: improvement}. In Section~\ref{sec:Carleson} we introduce the basics on radial doubling weights.
    Finally, we introduce the following notation, we will use $a\lesssim b$ if there exists a constant
	$C=C(\cdot)>0$ such that $a\leq C b$, and $a \gtrsim b$ is understood in an analogous manner. In particular, if $a\lesssim b$ and $a \gtrsim b$, then we write $a\asymp b$ and say that $a$ and $b$ are comparable.
	\section{Improvement of the Hardy-Littlewood's inequality  \eqref{HL}}\label{sec:2}
Some notation is needed before proving the main results of this section.
	The radial maximal function of $f\in\H(\D)$ is 
	$$
	R(f)(z) =  \sup_{0\leq \rho< |z|} \abs{f(\rho  e^{i\t})}, \quad z=|z|e^{i\t}\in \overline{\D}.  
	$$
Aiming to prove Theorem~\ref{th1: improvement} we get the following result of its own interest. 
	\begin{proposition}\label{prop: radial bound integral}
		Let $0<p<q<\infty$. Then,  there exists a constant $C=C(p,q)>0$ such that 
		\begin{equation}\label{eq:radmax}
		M_q(r,f)\leq \frac{C}{(\rho-r)^{\frac{1}{p}-\frac{1}{q}}} \nm{ \left\{\left (R(f)_{[p]}(\rho) \right )_{N(r,\rho),l} \right\}_{l=0}^{N(r,\rho)-1}}_{\ell^q},\,\, 0\leq r <\rho \leq 1,\, f\in\H(\D),
		\end{equation}
where $N(r,\rho)=E\left ( \frac{1}{\rho-r}\right )$.
	\end{proposition}
	\begin{proof}
		Let $0\leq r < \rho \leq 1$, $f\in \H(\D)$ and $N=N(r,\rho)=E\left (\frac{1}{\rho-r} \right)$. Asumme that the right hand side of \eqref{eq:radmax} is finite.  If $\frac{1}{2}\leq \rho \leq 1$, let $\tilde{r}=\frac{\rho+r}{2}$ and $\tilde{\t}_l\in[0,2\pi]$ such that $\max\limits_{\t\in I_{N,l}}\abs{f(\tilde{r}e^{i\t})} = \abs{f(\tilde{r}e^{i\tilde{\t}_l})}$. Then,
		\begin{equation}\label{eq: eq1 mp subharm}
			M_q^q(r,f) \leq M_q^q(\tilde{r},f) = \sum_{l=0}^{N-1}\int_{I_{N,l}}\abs{f(\tilde{r}e^{i\t})}^q \frac{d\theta}{2\pi} \lesssim (\rho-r)\sum_{l=0}^{N-1} \abs{f(\tilde{r}e^{i\tilde{\t}_l})}^q.
		\end{equation}
		On the other hand, there exists an absolute constant $K>2$ such that 
		$$
		D\left (\tilde{r}e^{i\tilde{\t}_l}, \frac{\rho-r}{K} \right) \subset \bigg\{z\in \D: r\leq z \leq \rho, \arg{z}\in I_{N,l-1}\cup I_{N,l}\cup I_{N,l+1} \bigg\},
		$$
		where  $I_{N,-1}=I_{N,N-1}$ and $I_{N,N}=I_{N,0}$. 
Let us denote $J_{N,l}=I_{N,l-1}\cup I_{N,l}\cup I_{N,l+1}$. Then, 
		\begin{equation*}
			\begin{split}
				\abs{f(\tilde{r}e^{i\tilde{\t}_l})}^q &= \left (\abs{f(\tilde{r}e^{i\tilde{\t}_l})}^p\right )^{\frac{q}{p}} \lesssim \left (\frac{1}{(\rho-r)^2}\int_{D\left (\tilde{r}e^{i\tilde{\t}_l},\frac{\rho-r}{K}\right )}\abs{f(z)}^pdA(z) \right )^{\frac{q}{p}} \\
				&\lesssim \left (\frac{1}{(\rho-r)^2}\int_r^\rho \int_{J_{N,l}} \abs{f(xe^{i\t})}^p\frac{d\theta}{2\pi} \,dx \right )^{\frac{q}{p}} \\
				&\lesssim \frac{1}{(\rho-r)^{\frac{q}{p}}}\left ( \int_{J_{N,l}} R(f)(\rho e^{i\t})^p\frac{d\theta}{2\pi}  \right )^{\frac{q}{p}}, \quad l=0,\ldots,N-1.
			\end{split}
		\end{equation*}
		Joining this inequality with \eqref{eq: eq1 mp subharm} it follows that
		\begin{equation}
			\begin{split}\label{casomayorunmedio}
				M_q^q (r,f)&\lesssim \frac{1}{(\rho-r)^{\frac{q}{p}-1}} \sum_{l=0}^{N-1} \left ( \int_{J_{N,l}} R(f)(\rho e^{i\t})^p\frac{d\theta}{2\pi}  \right )^{\frac{q}{p}} \\
				&\lesssim
\frac{1}{(\rho-r)^{\frac{q}{p}-1}} \nm{ \left\{\left (R(f)_{[p]}(\rho) \right )_{N,l} \right\}_{l=0}^{N-1}}_{\ell^q}^q, \quad 0\le r<\rho, \quad \frac{1}{2}\le \rho\le 1.
			\end{split}
		\end{equation}
\par		Assume now that $0<\rho<\frac{1}{2}$. By replacing $\rho$ by $1$, $f$ by $f_\rho$ and $r$ by $\frac{r}{\rho}$ in  \eqref{casomayorunmedio}, it follows that 
		\begin{equation}
			\begin{split}\label{casomenorunmedio}
		M_q(r,f)\lesssim \frac{\rho^{\frac{1}{p}-\frac{1}{q}}}{(\rho-r)^{\frac{1}{p}-\frac{1}{q}}}\nm{ \left\{\left (R(f)_{[p]}(\rho) \right )_{\tilde{N},\tilde{l}} \right\}_{\tilde{l}=0}^{\tilde{N}-1}}_{\ell^q},
		\end{split}
		\end{equation}
		where $\tilde{N}=E\left(\frac{\rho}{\rho-r}\right)$. If $N=\tilde{N}$ the result holds trivially.  On the other hand, if $N>\tilde{N}$ then
		$$
		\tilde{N}=E\left (\frac{\rho}{\rho-r}\right )\geq \frac{1}{2} \frac{\rho}{\rho-r} \geq \frac{\rho}{2} E\left (\frac{1}{\rho-r} \right)=\frac{\rho}{2}N.
		$$
		Therefore,  for every $l=0,\ldots, N-1$ and $\tilde{l}=0,\ldots, \tilde{N}-1$, we have that $|I_{N,l}|<|I_{\tilde{N},\tilde{l}}|\leq \frac{2}{\rho}|I_{N,l}|$. Then, we can cover every interval $I_{\tilde{N},\tilde{l}}$ with at most $E\left (\frac{2}{\rho} \right)+1$ intervals $I_{N,l}$. It means, for each $\tilde{l}$, there exists a $l_{\tilde{l}}$ such that
		\begin{equation}\label{eq: trick scale}
		I_{\tilde{N},\tilde{l}} \subset \bigcup_{j=0}^{E\left (\frac{2}{\rho} \right)} I_{N,l_{\tilde{l}}+j}.
		\end{equation}
So,
		\begin{equation}
			\begin{split}\label{eq:2casomenorunmedio}
				\nm{ \left\{\left (R(f)_{[p]}(\rho) \right )_{\tilde{N},\tilde{l}} \right\}_{\tilde{l}=0}^{\tilde{N}-1}}_{\ell^q}^q &\leq \sum_{\tilde{l}=0}^{\tilde{N}-1}\left (\sum_{j=0}^{E\left (\frac{2}{\rho} \right)} \int_{I_{N,l_{\tilde{l}}+j}} R(f)(\rho e^{i\t})^p\frac{d\theta}{2\pi} \right)^{\frac{q}{p}} \\
				&\lesssim \frac{1}{\rho^{\frac{q}{p}-1}}\sum_{\tilde{l}=0}^{\tilde{N}-1}\sum_{j=0}^{E\left (\frac{2}{\rho} \right)} \left ( \int_{I_{N,l_{\tilde{l}}+j}} R(f)(\rho e^{i\t})^p\frac{d\theta}{2\pi} \right)^{\frac{q}{p}} \\
				&\lesssim \frac{1}{\rho^{\frac{q}{p}-1}} \nm{ \left\{\left (R(f)_{[p]}(\rho) \right )_{N,l} \right\}_{l=0}^{N-1}}_{\ell^q}^q,
			\end{split}
		\end{equation}
where in the last inequality we have used that each interval $I_{N,l}$ appears at most two times in the sum $\sum_{\tilde{l}=0}^{\tilde{N}-1}\sum_{j=0}^{E\left (\frac{2}{\rho} \right)} \left ( \int_{I_{N,l_{\tilde{l}}+j}} R(f)(\rho e^{i\t})^p\frac{d\theta}{2\pi} \right)^{\frac{q}{p}}$.  Joining \eqref{casomenorunmedio} and \eqref{eq:2casomenorunmedio}, we get
\begin{equation*}
			\begin{split}
				M_q^q (r,f) \lesssim
\frac{1}{(\rho-r)^{\frac{q}{p}-1}} \nm{ \left\{\left (R(f)_{[p]}(\rho) \right )_{N,l} \right\}_{l=0}^{N-1}}_{\ell^q}^q, \quad 0\le r<\rho< \frac{1}{2}.
\end{split}
		\end{equation*}
		Concluding the proof.
	\end{proof}
	The next step to obtain a proof of \eqref{eq: th improve} for $0<q<\infty$ will be to replace the radial maximal operator on the right-hand side of \eqref{eq:radmax} by the function $f$ itself. This inequality leads us to consider the boundedness of the radial maximal function on an amalgam type space. 
	
	Let $1\leq p,q<\infty$ the classical amalgam space $\ell^q(L^p)$ is the space of measurable functions $f$ on $\R$ such that
	$$
	\nm{f}_{\ell^q(L^p)}^q = \sum_{n\in \Z} \left (\int_n^{n+1}\abs{f(x)}^p dx\right )^{\frac{q}{p}}<\infty.
	$$
	These spaces where introduced by Wiener in \cite{Wiener} and  have arisen in various areas of mathematical analysis such as  Tauberian theorems, Fourier multipliers or approximation theory.  See \cite{FS} and the references therein for a detailed study and applications
of these spaces. We also consider the Hardy-Littlewood maximal operator 
	$$
	M(f)(x)=\sup_{x\in I}\frac{1}{|I|}\int_I \abs{f(y)}dy,\quad f\in L^1_{loc}(\R),\quad x\in\R,
	$$
	where the supremum is taken over all the finite closed interval $I$ such that $x\in I$ and $L^1_{loc}(\R)$ is the space of measurable functions that are integrable on any compact subset of $\R$. The boundedness of $M$ on $\ell^q(L^p)$ follows from \cite[Theorem 4.2 and Theorem~4.5]{Carton-Lebrun} (see also \cite[Theorem 5]{AguilOrtega}).
	\begin{lettertheorem}\label{th: maximal acotada amalgama}
		Let $1<p,q<\infty$. Then,  the Hardy-Littlewood maximal operator is bounded on $\ell^q(L^p)$.
	\end{lettertheorem}
	In the next result, we obtain the boundedness of the radial maximal operator on an amalgam type space as a byproduct of the boundedness of the Hardy-Littlewood maximal operator on the classical amalgam space.
	\begin{lemma}\label{lemma: radial maximal bounded}
		Let $0<p<q<\infty$. Then,  there exists a constant $C=C(p,q)>0$ such that
		$$
		\nm{ \left\{\left (R(f)_{[p]}(1) \right )_{N,l} \right\}_{l=0}^{N-1}}_{\ell^q} \leq C \nm{ \left \{\left (f_{[p]}(1) \right )_{N,l} \right\}_{l=0}^{N-1}}_{\ell^q},\quad f\in H^p,\, N\in \N.
		$$
	\end{lemma}
	\begin{proof}
		For a measurable function $f$ on $\T$ denote 
		$$
		M_{\T}(f)(\t)=\sup_{0\leq |T|\leq \pi} \frac{1}{T}\int_0^T | f(e^{i(\t+t)})|\, dt,\quad \t\in[0,2\pi].
		$$
		 By  the proof of \cite[Theorem 1.8]{Duren}, 
		$$
		\abs{f(z)}^{\frac{p}{2}} \leq P[\abs{f}^{\frac{p}{2}}](z) \leq R(P[\abs{f}^{\frac{p}{2}}])(e^{i\t}) \leq 2 M_{\T}(\abs{f}^\frac{p}{2})(\t), \quad z=re^{i\t}\in\D,
		$$
where $P$ denotes the Poisson integral.
		So $R(f)(e^{i\t})^p\leq 4 M_{\T}(\abs{f}^\frac{p}{2})(\t)^2$, which implies that
		\begin{equation}\label{eq: eq1 radial bounded hormander}
			\begin{split}
				\nm{ \left \{\left (R(f)_{[p]}(1) \right )_{N,l} \right\}_{l=0}^{N-1}}_{\ell^q}^q\lesssim \sum_{l=0}^{N-1} \left ( \int_{I_{N,l}} M_\T(\abs{f}^\frac{p}{2})(\t)^2\frac{d\theta}{2\pi} \right )^{\frac{q}{p}}.
			\end{split}
		\end{equation}
		Now consider the function $$F(\t)=\begin{cases}
		 \abs{f(e^{i\t})}^\frac{p}{2}	&\text{if } \t\in[-\pi,3\pi]\\
		 0 &\text{otherwise}. 
		\end{cases}$$ 
Then, a simple computation shows that $M_\T(\abs{f}^\frac{p}{2})(\t) \leq M(F)(\t)$. Joining this inequality  with \eqref{eq: eq1 radial bounded hormander} and denoting $\tilde{q}=2\frac{q}{p}>2$,  we obtain
		\begin{equation}\label{eq: eq2 radial bounded maximal}
			\nm{ \left \{\left (R(f)_{[p]}(1) \right )_{N,l} \right\}_{l=0}^{N-1}}_{\ell^q}^q\lesssim \sum_{l\in \Z} \left ( \int_{I_{N,l}} M(F)(\t)^2\frac{d\theta}{2\pi} \right )^{\frac{\tilde{q}}{2}},
		\end{equation}
		here $I_{N,l}=\left [2\pi\frac{l}{N},2\pi\frac{l+1}{N} \right ]$ for every $l\in\Z$. \\
		
		For the next step, denote $F_\delta(x)=F(\delta x)$, $\delta>0$, and observe that
$M(F_\delta)(x) = M(F)(\delta x)$, $x\in\R$. So, making the change of variables $x=\frac{\t}{\delta_N}$ where $\delta_N=\frac{2\pi}{N}$ and  applying Theorem~\ref{th: maximal acotada amalgama}, it follows that
		\begin{equation}\label{eq: eq3 bounded maximal}
			\begin{split}
				&\sum_{l\in \Z} \left ( \int_{I_{N,l}} M(F)(\t)^2\frac{d\theta}{2\pi} \right )^{\frac{\tilde{q}}{2}} 
\\ &= \frac{ \delta_N^{\frac{\tilde{q}}{2}}}{(2\pi)\frac{\tilde{q}}{2}} \sum_{l\in \Z} \left ( \int_{l}^{l+1} M(F_{\delta_N})(x)^2 dx \right )^{\frac{\tilde{q}}{2}} \lesssim \frac{ \delta_N^{\frac{\tilde{q}}{2}}}{(2\pi)\frac{\tilde{q}}{2}} \sum_{l\in \Z} \left ( \int_{l}^{l+1} F_{\delta_N}(x)^2 dx \right )^{\frac{\tilde{q}}{2}} \\
				&= \sum_{l\in \Z} \left ( \int_{I_{N,l}} F(\t)^2 \frac{d\theta}{2\pi}\right )^{\frac{\tilde{q}}{2}} = \sum_{l=-E[\frac{N}{2}]-1}^{E[\frac{3}{2}N]}\left ( \int_{I_{N,l}} F(\t)^2 \frac{d\theta}{2\pi}\right )^{\frac{\tilde{q}}{2}}.
			\end{split}
		\end{equation}
Finally, since $\abs{f(e^{i\cdot})}^{\frac{p}{2}}$ is $2\pi$-periodic function
		$$
		\sum_{l=-E[\frac{N}{2}]-1}^{E[\frac{3}{2}N]}\left ( \int_{I_{N,l}} F(\t)^2 \frac{d\theta}{2\pi}\right )^{\frac{\tilde{q}}{2}} \leq 3 \sum_{l=0}^{N-1}\left ( \int_{I_{N,l}} \abs{f(e^{i\t})}^p \frac{d\theta}{2\pi}\right )^{\frac{\tilde{q}}{2}}\asymp
 \nm{ \left \{\left (f_{[p]}(1) \right )_{N,l} \right\}_{l=0}^{N-1}}_{\ell^q}^q.
		$$
		The proof ends by joining this last inequality with \eqref{eq: eq2 radial bounded maximal} and \eqref{eq: eq3 bounded maximal}.
	\end{proof}
\vspace{1em}
	Now we can deal with the proof of Theorem~\ref{th1: improvement}.
	
\subsection{Proof of Theorem~\ref{th1: improvement}. }

We begin with the proof of the inequality \eqref{eq: th improve}.		
		Let $0\leq r < \rho \leq 1$, $f\in \H(\D)$ 
and $N=N(r,\rho)=E \left(\frac{1}{\rho-r} \right)$. Assume that the right hand side of \eqref{eq: th improve} is finite. If $0<q<\infty$ the statement follows by joining Proposition~\ref{prop: radial bound integral} and Lemma~\ref{lemma: radial maximal bounded} applied to $f_\rho$
		$$
		M_q(r,f) \lesssim \frac{1}{(\rho-r)^{\frac{1}{p}-\frac{1}{q}}} \nm{ \left \{\left (R(f)_{[p]}(\rho) \right )_{N,l} \right\}_{l=0}^{N-1}}_{\ell^q} \lesssim  \frac{1}{(\rho-r)^{\frac{1}{p}-\frac{1}{q}}} 
\nm{ \left \{\left (f_{[p]}(\rho) \right )_{N,l} \right\}_{l=0}^{N-1}}_{\ell^q}^q.
		$$
		
		If  $q=\infty$, assume first that $\rho\geq \frac{1}{2}$ and take $\tilde{r}=\frac{\rho+r}{2}$. Let $\tilde{\t}\in[0,2\pi]$ such that $\abs{f(\tilde{r}e^{i\tilde{\t}})}=M_\infty(\tilde{r},f)$, then 
		\begin{equation}\label{eq: eq1 minfinity}
		M_\infty^p(r,f)\leq M_\infty^p(\tilde{r},f) =\abs{f(\tilde{r}e^{i\tilde{\t}})}^p \leq \int_0^{2\pi} \frac{1-(\frac{\tilde{r}}{\rho})^2}{\abs{1-\frac{\tilde{r}}{\rho}e^{i(\t-\tilde{\t})}}^2}\abs{f(\rho e^{i\t})}^p \frac{d\theta}{2\pi}.
		\end{equation}
		On the other hand, applying the definition of $\tilde{r}$ and the fact that $\rho\geq \frac{1}{2}$ we get
		$$
		1-\left(\frac{\tilde{r}}{\rho}\right)^2 \asymp \rho-r,\quad \text{ and } \quad \frac{\tilde{r}}{\rho} \geq \frac{1}{2},
		$$
		then 
		$$
		\int_0^{2\pi} \frac{1-(\frac{\tilde{r}}{\rho})^2}{\abs{1-\frac{\tilde{r}}{\rho}e^{i(\t-\tilde{\t})}}^2}\abs{f(\rho e^{i\t})}^p \frac{d\theta}{2\pi} \lesssim \int_0^{2\pi} \frac{\rho-r}{((\rho-r)+\abs{\t-\tilde{\t}})^2}\abs{f(\rho e^{i\t})}^p \frac{d\theta}{2\pi}.
		$$
		
		Let $\tilde{l}\in \{0,\ldots N\} $ such that $\tilde{\t}\in I_{N,\tilde{l}}$.  Assume without loss of generality that $0<\tilde{l}<N-1$, otherwise the following calculations are similar. Then, bearing in mind the definition of $I_{N,l}$, $\rho-r\asymp \frac{1}{N}$, \eqref{eq: eq1 minfinity} and previous inequality
		\begin{equation}
			\begin{split}\label{eq:infty1}
				M_\infty^p(r,f) &\lesssim \sum_{l=0}^{\tilde{l}-1} \int_{I_{N,l}} \frac{\rho-r}{((\rho-r)+(\tilde{\t}-\t))^2}\abs{f(\rho e^{i\t})}^p\frac{d\theta}{2\pi}\\
				&+ \int_{I_{N, \tilde{l}}} \frac{\rho-r}{((\rho-r)+\abs{\t-\tilde{\t}})^2}\abs{f(\rho e^{i\t})}^p \frac{d\theta}{2\pi} + \sum_{l=\tilde{l}}^{N-1} \int_{I_{N,l}} \frac{\rho-r}{((\rho-r)+(\t-\tilde{\t}))^2}\abs{f(\rho e^{i\t})}^p\frac{d\theta}{2\pi} \\
				&\lesssim  \sum_{l=0}^{\tilde{l}-1} \frac{1}{(\rho-r)(\tilde{l}-l)^2}\int_{I_{N,l}} \abs{f(\rho e^{i\t})}^p\frac{d\theta}{2\pi}
				+ \frac{1}{(\rho-r)}\int_{I_{N, \tilde{l}}} \abs{f(\rho e^{i\t})}^p\frac{d\theta}{2\pi} \\ &+ \sum_{l=\tilde{l}+1}^{N-1} \frac{1}{(\rho-r)(l-\tilde{l})^2}\int_{I_{N,l}} \abs{f(\rho e^{i\t})}^p \frac{d\theta}{2\pi} \leq \frac{\left (1+ 2\sum_{n=1}^{\infty}\frac{1}{n^2} \right)}{\rho-r} 
\nm{ \left \{\left (f_{[p]}(\rho) \right )_{N,l} \right \}_{l=0}^{N-1}}_{\ell^\infty}^p\\ &\lesssim \frac{1}{\rho-r}\nm{ \left \{\left (f_{[p]}(\rho) \right )_{N,l} \right\}_{l=0}^{N-1}}_{\ell^\infty}^p, \quad \frac{1}{2}\le \rho<1.
			\end{split}
		\end{equation}
		concluding the proof of this case. 
		
		Now, for $0<\rho<\frac{1}{2}$,  by replacing $\rho$ by $1$, $f$ by $f_\rho$ and $r$ by $\frac{r}{\rho}$ in \eqref{eq:infty1}, we obtain
		\begin{equation}\label{eq: eq2 minfinity scale1}
		M_\infty(r,f) \lesssim  \frac{\rho^{\frac{1}{p}}}{(\rho-r)^{\frac{1}{p}}}\nm{\{\left (f_{[p]}(\rho) \right )_{\tilde{N},\tilde{l}}\}_{\tilde{l}=0}^{\tilde{N}-1}}_{\ell^\infty}
		\end{equation}
		where $\tilde{N}=E \left (\frac{\rho}{\rho-r} \right)$. If $N=\tilde{N}$, the proof is finished.  If $N>\tilde{N}$,  using same notation of the proof of Proposition~\ref{prop: radial bound integral},  \eqref{eq: trick scale} and taking supremum on $l$
		$$
		\nm{ \left \{\left (f_{[p]}(\rho) \right )_{\tilde{N},\tilde{l}} \right\}_{\tilde{l}=0}^{\tilde{N}-1}}_{\ell^\infty} \lesssim \left (\frac{1}{\rho}\right )^\frac{1}{p} 
\nm{ \left \{\left (f_{[p]}(\rho) \right )_{N,l} \right\}_{l=0}^{N-1}}_{\ell^\infty}.
		$$
Therefore, 
\begin{equation}\label{eq: eq2 minfinity scale2}
M_\infty^p(r,f) \lesssim  \frac{1}{\rho-r}\nm{ \left \{\left (f_{[p]}(\rho) \right )_{N,l} \right\}_{l=0}^{N-1}}_{\ell^\infty}^p, \quad 0\le \rho<\frac{1}{2}.
\end{equation}
Joining \eqref{eq: eq2 minfinity scale1} and \eqref{eq: eq2 minfinity scale2} the proof of \eqref{eq: th improve} is finished.
	
	\vspace{1em}
	\par In order to prove \eqref{eq:improve1},  
for every $\lambda>1$ we denote by 
	$$
	\mathcal{L}(\lambda) = \set{f(z)=\sum_{k=0}^\infty a_k z^{n_k}\in \H(\D): n_{k+1}\geq \lambda n_k},
	$$
	the set of all lacunary series with step $\lambda$. The following lemma plays a key role in the proof of \eqref{eq:improve1}.
	
	\begin{letterlemma}\cite[Lemma 6.5]{zygmund}\label{lemma: lacunary}
		Let $E\subset[0,2\pi]$ with $\abs{E}>0$, $\lambda>1$ and $A>1$. Then, there exists a $N_0(A,\lambda,\abs{E})\in\N$ such that for any $f(z)=\sum_{k=0}^\infty a_k z^{n_k}\in \mathcal{L}(\lambda)$ with $n_0\geq N_0$ 
		$$
		\frac{\abs{E}}{2\pi A}\nm{f}_{H^2}^2 \leq  \int_E \abs{f(e^{i\t})}^2\frac{d\theta}{2\pi} \leq \frac{A\abs{E}}{2\pi}\nm{f}_{H^2}^2
		$$
	\end{letterlemma}

	Now we can deal with the rest of the proof of Theorem~\ref{th1: improvement}.  
		Firstly, 
observe that
		$$
				\nm{ \left \{\left (f_{[p]}(1) \right )_{N,l} \right\}_{l=0}^{N-1}}_{\ell^\infty} = \left (\sup_{j=0,\ldots, N-1}\int_{I_{N,l}} \abs{f(e^{i\t})}^p\frac{d\theta}{2\pi}  \right )^\frac{1}{p} \leq 
\nm{f}_{H^p}.
		$$
	Moreover,	if $q<\infty$ then
		\begin{equation*}
			\begin{split}
				\nm{ \left\{\left (f_{[p]}(1) \right )_{N,l} \right\}_{l=0}^{N-1}}_{\ell^q}^q &\leq\nm{ \left\{\left (f_{[p]}(1) \right )_{N,l} \right\}_{l=0}^{N-1}}_{\ell^\infty}^{q-p}\sum_{l=0}^{N-1} \left (f_{[p]}(1) \right )_{N,l}^p \\
				&\leq   \nm{f}_{H^p}^{q-p}  \sum_{l=0}^{N-1} \int_{I_{N,l}} \abs{f(e^{i\t})}^p\frac{d\theta}{2\pi} = \nm{f}_{H^p}^q,
			\end{split}
		\end{equation*}
		so  \eqref{eq:trivial} holds.
		
		Now, to prove \eqref{eq:improve1} it is enough to see that for each $N\in\N$, there exists a function $f_N\in H^p$ such that
		\begin{equation}\label{eq: eq1 improve hp}
		\frac{\nm{f_N}_{H^p}}{	\nm{ \left\{\left ((f_N)_{[p]}(1) \right )_{N,l} \right\}_{l=0}^{N-1}}_{\ell^q}} \geq C N^{\frac{1}{p}-\frac{1}{q}},
		\end{equation}
		where $C=C(p)>0$. 
		
		First assume that $2=p<q\leq \infty$. Taking  $\lambda=A=2$ and $\abs{E}=\frac{2\pi}{N}$ in Lemma~\ref{lemma: lacunary},  there exists $N_0(N)\in \N$ and $f_N(z)=\sum_{k=0}^\infty a_k z^{2^{k+k_0(N)}}\in H^2\cap \mathcal{L}(2)$ with $2^{k_0(N)}\geq N_0(N)$ such that
		$$
		\frac{1}{2N}\nm{f_N}_{H^2}^2 
\leq \int_{E_N}\abs{f_N(e^{i\t})}^2\frac{d\theta}{2\pi} \leq 
\frac{2}{N}\nm{f_N}_{H^2}^2,
		$$
		for every $E_N\subset [0,2\pi]$ such that $\abs{E_N}=\frac{2\pi}{N}$. In particular, for every $l=0,\ldots, N-1$,
		$$
		\frac{1}{N^{\frac{1}{2}}}\nm{f_N}_{H^2}  \leq \left ((f_N)_{[2]}(1) \right )_{N,l} \leq \frac{2^{\frac{1}{2}}}{N^{\frac{1}{2}}}\nm{f_N}_{H^2}^2.
		$$
		Then, on the one hand, if $q=\infty$
		$$
		\nm{ \left\{\left ((f_N)_{[2]}(1) \right )_{N,l} \right\}_{l=0}^{N-1}}_{\ell^\infty} = \sup_{l=0,\ldots ,N-1} \left ((f_N)_{[2]}(1) \right )_{N,l} \leq \frac{2^{1/2}}{ N^{\frac{1}{2}}}\nm{f_N}_{H^2},
		$$
		on the other hand, if $q<\infty$
		\begin{equation*}
			\begin{split}
				\nm{ \left\{\left ((f_N)_{[2]}(1) \right )_{N,l} \right\}_{l=0}^{N-1}}_{\ell^q} &= \left (\sum_{l=0}^{N-1} \left ((f_N)_{[2]}(1) \right )_{N,l}^q \right )^{\frac{1}{q}}\\ 
				&\leq \frac{2^{\frac{1}{2}}}{N^{\frac{1}{2}}} \nm{f_N}_{H^2} \left ( \sum_{l=0}^{N-1} 1 \right)^{\frac{1}{q}} = \frac{2^{1/2}}{N^{\frac{1}{2}-\frac{1}{q}}}\nm{f_N}_{H^2}.
			\end{split}
		\end{equation*}
		In conclusion, \eqref{eq: eq1 improve hp} holds with $C=\frac{1}{2^{1/2}}$.
		
		Now if $p\neq 2$, for each $N\in \N$, let be $f_N\in H^2$  such that \eqref{eq: eq1 improve hp} holds for $p=2$. Let us consider  its factorization
		$$
		f_N=g_N B_N,
		$$
		where $B_N$ is the Blaschke product of the zeros of $f_N$ and $g_N\in H^2$ never vanishes in $\D$.  So, there exists an analytic function $F_{N,p}$ such that $F_{N,p}=g_N^{\frac{2}{p}}$. Moreover, 
		\begin{equation}\label{eq: eq2 norm equi fact}
		\nm{F_{N,p}}_{H^p}^p=\nm{g_N}_{H^2}^2=\nm{f_N}_{H^2}^2,
		\end{equation}
		and
		\begin{equation}\label{eq: eq3 norm equi fact}
			\begin{split}
				\nm{ \left\{\left ((F_{N,p})_{[p]}(1) \right )_{N,l} \right\}_{l=0}^{N-1}}_{\ell^q} &= \nm{\left  \{\left ((f_{N})_{[2]}(1) \right )_{N,l}^{\frac{2}{p}}\right \}_{l=0}^{N-1}}_{\ell^q},
			\end{split}
		\end{equation}
		then joining \eqref{eq: eq2 norm equi fact} with \eqref{eq: eq3 norm equi fact} and arguing as in case $p=2$, we obtain for $q=\infty$
		$$
		\nm{ \left\{\left ((F_{N,p})_{[p]}(1) \right )_{N,l} \right\}_{l=0}^{N-1}}_{\ell^\infty}= \sup_{l=0,\ldots,N-1} \left ((f_{N})_{[2]}(1) \right )_{N,l}^{\frac{2}{p}}  \leq \frac{2^{\frac{2}{p}}}{N^{\frac{1}{p}}} \nm{f_N}_{H^2}^{\frac{2}{p}} = \frac{2^{\frac{2}{p}}}{N^{\frac{1}{p}}} \nm{F_{N,p}}_{H^p}.
		$$
		On the other hand, if $q<\infty$, then
		\begin{equation*}
			\begin{split}
					\nm{\left \{\left ((F_{N,p})_{[p]}(1) \right )_{N,l} \right\}_{l=0}^{N-1}}_{\ell^q} &= \left (\sum_{l=0}^{N-1}\left ((f_{N})_{[2]}(1) \right )_{N,l}^{\frac{2q}{p}} \right )^{\frac{1}{q}}  \\
					&\leq \frac{2^{\frac{2}{p}}}{N^{\frac{1}{p}}} \nm{f_N}_{H^2}^{\frac{2}{p}} \left ( \sum_{l=0}^{N-1} 1 \right)^{\frac{1}{q}} = \frac{2^{\frac{2}{p}}}{N^{\frac{1}{p}-\frac{1}{q}}} \nm{F_{N,p}}_{H^p}.
			\end{split}
		\end{equation*}
		In both cases, we conclude that \eqref{eq: eq1 improve hp} holds with $C=\frac{1}{2^{\frac{2}{p}}}$, and the functions $F_{N,p}$. This finishes the proof.

	\section{A Carleson measure type problem} \label{sec:Carleson}
	In this section, we will use different properties of $\DD$ that are summarized in the following lemma. See \cite[Lemma~2.1]{PelSum14}.
	\begin{letterlemma}
		\label{lemma: caract dgorro}
		Let $\om$ be a radial weight. Then, the following statements are equivalent:
		\begin{itemize}
			\item[\rm(i)] $\om \in \DD$;
			\item[\rm(ii)] There exist $C=C(\om)\geq 1$ and $\a_0=\a_0(\om)>0$ such that
			$$ \omg(s)\leq C \left(\frac{1-s}{1-t}\right)^{\a}\omg(t), \quad 0\leq s\leq t<1,$$
			for all $\a\geq \a_0$;
			\item[\rm(iii)] There exists $C=C(\om)>0$ and $\lambda=\lambda(\om)>0$ such that
			$$
			\int_0^r\left (\frac{1-r}{1-s} \right)^\lambda \om(s)ds \leq C \omg(r),\quad 0\leq r < 1.
			$$
		\end{itemize}
	\end{letterlemma}
	We also need the following technical result.
	\begin{lemma}\label{lemma: weight integral radius}
		Let $\om\in\DDD$ and $K\in \N\setminus\{1\}$ such that \eqref{eq: def Dcheck} holds. Then, there exist positive constants $C_1 ,C_2$ such that
		$$
		C_1 \omg \left ( 1-\frac{1-r}{K}\right) \leq \int_r^{1-\frac{1-r}{K}}\om(s)ds \leq \omg(r) \leq C_2  \omg \left ( 1-\frac{1-r}{K}\right ),\quad 0\leq r <1.
		$$
	\end{lemma}
	\begin{proof}
		Since $\om\in \Dd$, if  $C>1$ is that from \eqref{eq: def Dcheck}. Then,
		$$
		\int_r^{1-\frac{1-r}{K}}\om(s)ds = \omg(r)-\omg \left ( 1-\frac{1-r}{K} \right ) \geq (C-1)\omg\left ( 1-\frac{1-r}{K} \right ), \quad 0\leq r <1.
		$$
		On the other hand, if $\alpha>0$ and $C>0$ are the ones from Lemma~\ref{lemma: caract dgorro} (ii), then
		$$
		\int_r^{1-\frac{1-r}{K}}\om(s)ds \leq \omg(r) \leq C K^\a \omg\left ( 1-\frac{1-r}{K} \right ),\quad 0\leq r <1.
		$$
This finishes the proof.
	\end{proof}
	
En route to the proof of the equivalence (i)$\Leftrightarrow$(ii) of Theorems~\ref{th2: charact Tg} and \ref{th2: charact SgMg}, we  prove the following result which generalizes Luecking's assertion on the equivalence between   \eqref{eq: carleson-type inequality} and \eqref{eq:Luecking-discrete} for $\om=1$.
\begin{proposition}\label{propo: carleson type}
		Let $0<p,q,s,t<\infty$, $\om\in\DDD$, $n\in\N\cup\{0\}$ and $K=K(\om)\in \N \setminus \{1\}$ such that \eqref{eq: def Dcheck} holds. If $\mu_r$ are positive Borel measures on $[0,2\pi]$ varying on $0\leq r<1$ and $\nu$ a positive Borel measure on $[0,1)$. Then, the following conditions  are equivalent:
	\begin{itemize}
		\item[(i)] There exists a constant $C>0$ such that the Carleson-type inequality \eqref{eq: carleson-type inequality} holds;
		\item[(ii)] There exists a constant $C>0$ such that
		\begin{equation}\label{eq: carleson-type discrete}
		\left (\sum_{j=1}^\infty \int_{r_{j-1}}^{r_j}K^{jt(n+\frac{1}{p})}\omg(r_j)^{-\frac{t}{q}} \left (\sum_{l=0}^{K^{j+2}-1} \abs{a_{j,l}}^{s} \mu_r(I_{K^{j+2},l}) \right)^{\frac{t}{s}} d\nu(r)\right )^{\frac{1}{t}} \leq C \nm{a_{j,l}}_{\ell^{p,q}}.
		\end{equation}
	\end{itemize}
	Moreover,
	$$
	\inf\{C>0: \text{ such that } \eqref{eq: carleson-type inequality} \text{ holds} \} \asymp \inf\{C>0: \text{ such that } \eqref{eq: carleson-type discrete} \text{ holds} \}.
	$$
	\end{proposition}

We need some notation and recall some results to prove Proposition~\ref{propo: carleson type}.	We write $\rho(a,z)=\left |\frac{a-z}{1-\overline{a}z} \right |$ for the pseudohyperbolic distance between $z,a\in \D$. A sequence $\{z_k\}_{k=0}^\infty$ in $\D$ is called separated if $\inf_{k\neq j}\rho(z_k,z_j)>0$. Now each sequence $\{z_k\}$ in $\D$, is re-indexed in the following way depending on $\om$. Let $K\in\N\setminus\{1\}$ such that \eqref{eq: def Dcheck} holds, then for each $j\in\N$, let $\{z_{j,l}\}_l$ denote the points of the sequence $\{z_k\}$ in the annulus $A_j =A_j(K)=\{z: r_{j-1} \leq |z| < r_j \}$, where $r_j=r_j(K)=1-K^{-j}$.
	\begin{lettertheorem}\label{th: atomic kian}\cite[Theorem 1]{PelRatSierra}
		Let $0<p\leq \infty$, $0<q<\infty$, $\om\in \DDD$ and $\{z_{j,l}\}$ a separated sequence in $\D$. Let $\a=\a(\om)>0$ and $\lambda=\lambda(\om)>0$ such that Lemma~\ref{lemma: caract dgorro} (ii) and (iii) holds. If $M>1+\frac{1}{p} + \frac{\a+\lambda}{q}$ and $\{a_{j,l}\}_{j,l}\in \ell^{p,q}$, then the function
		$$
		F(z)=\sum_{j,l}a_{j,l}\frac{\left (1-\abs{z_{j,l}} \right)^{M-\frac{1}{p}}\omg(z_{j,l})^{-\frac{1}{q}}}{(1-\overline{z_{j,l}}z)^M}\in\H(\D),
		$$
		and there exists a constant $C=C(M,\om,p,q)>0$ such that
		$$
		\nm{F}_{\Apq{p}{q}} \leq C \nm{\{a_{j,l}\}}_{\ell^{p,q}}.
		$$
	\end{lettertheorem}
	Throughout the rest of the paper, we will denote for each $j\in\N$ and  $l\in \{0, \dots K^{j+2}-1\}$,
\begin{equation}\label{eq:Qij}
Q_{j,l} = \{z\in\D : r_{j-1}\leq |z| < r_j, \arg z\in I_{K^{j+2},l} \} .
\end{equation}
	The next lemma follows from \cite[p. 743]{PelRatSierra} (see also \cite[Lemma 4]{PelRatSierra}) and will also play a key role in the proof of Proposition~\ref{propo: carleson type}.
	\begin{letterlemma} \label{lemma: descomposicion norma bloques}
		Let $0<p,q<\infty$, $n\in\N\cup\{0\}$, $\om\in \DDD$ and $K\in \N\setminus\{1\}$ such that \eqref{eq: def Dcheck} holds. Then,
		$$
		\nm{\left \{\abs{f^{(n)}_{j,l}}K^{-j(n+\frac{1}{p})}\omg(r_j)^{\frac{1}{q}}\right \}}_{\ell^{p,q}}\lesssim\nm{f}_{\Apq{p}{q}}, \quad f\in \H(\D),
		$$
		where $\abs{f_{j,l}^{(n)}} = \sup_{z\in Q_{j,l}}\abs{f^{(n)}(z)}$.
	\end{letterlemma}
	
	\begin{Prf}{\em{Proposition~\ref{propo: carleson type}}}
		$\mathbf{(i) \implies (ii)}$. \\
	
		Let be 
		$$z_{j,l}=\frac{r_j+r_{j-1}}{2}e^{2\pi \frac{l+\frac{1}{2}}{K^{j+2}}}, \quad j\in \N, \quad l=0,1,\ldots, K^{j+2}-1,$$
		the middle points of the rectangles $Q_{j,l}$, it is clear that $\{z_{j,l} \}$ is a separated sequence.
		Let $\{a_{j,l}\}$ be a positive and finite sequence and $\{r_{j,l}\}$ be Rademacher's functions (see \cite[Appendix A]{Duren}). Let us consider the functions
		$$
		F_{\tau}(z) = \sum_{j,l}a_{j,l}r_{j,l}(\tau)\frac{\left (1-\abs{z_{j,l}} \right)^{M-\frac{1}{p}}\omg(z_{j,l})^{-\frac{1}{q}}}{(1-\overline{z_{j,l}}z)^M},\quad \tau\in[0,1],
		$$
		where $M$ satisfies the hypothesis of  Theorem~\ref{th: atomic kian}. Then, by Theorem~\ref{th: atomic kian} each $F_{\tau}\in\H(\D)$
		and  $\nm{F_\tau}_{\Apq{p}{q}}\lesssim \nm{\set{a_{j,l}}}_{\ell^{p,q}}$. 
So, 
		$$
		\int_0^1 \left (\int_0^{2\pi} \abs{F_\tau^{(n)}(re^{i\t})}^s d\mu_r(\t) \right)^{\frac{t}{s}}d\nu(r) \lesssim C^t \nm{\set{a_{j,l}}}_{\ell^{p,q}}^t,\quad \tau\in[0,1],
		$$
where $C$ is that of \eqref{eq: carleson-type inequality}.
		Therefore,  by Fubini's theorem we obtain
		\begin{equation}\label{eq: kinchine 1}
			\int_0^1\int_0^1 \left (\int_0^{2\pi} \abs{F^{(n)}_\tau(re^{i\t})}^s d\mu_r(\t) \right)^{\frac{t}{s}}d\tau \,d\nu(r) \lesssim C^t \nm{\{a_{j,l}\}}_{\ell^{p,q}}^t.
		\end{equation}
		Now,  by 
Khinchine-Kahane-Kalton inequality (see \cite[Theorem 2.1]{kalton}), let $C(M,n)=M(M+1)\cdots(M+n-1)$
		\begin{equation}\label{eq: kinchine 2}
			\begin{split}
				&\int_0^1 \left (\int_0^{2\pi} \abs{F^{(n)}_\tau(re^{i\t})}^s d\mu_r(\t) \right)^{\frac{t}{s}}d\tau  \\ 
				& \asymp \int_0^1 \left (\int_0^{2\pi} \abs{\sum_{j,l}C(M,n)a_{j,l}\overline{z_{j,l}}^n r_{j,l}(\tau)\frac{(1-\abs{z_{j,l}})^{M-\frac{1}{p}}\omg(z_{j,l})^{-\frac{1}{q}}}{(1-\overline{z_{j,l}}re^{i\t})^{M+n}}}^s  d\mu_r(\t) \right )^{\frac{t}{s}}d\tau \\ &= \int_0^1  \nm{\sum_{j,l}C(M,n)a_{j,l}\overline{z_{j,l}}^nr_{j,l}(\tau)\frac{(1-\abs{z_{j,l}})^{M-\frac{1}{p}}\omg(z_{j,l})^{-\frac{1}{q}}}{(1-\overline{z_{j,l}}re^{i\cdot})^{M+n}}}^t_{L^s(\mu_r)} d\tau \\
				&\gtrsim \left (\int_0^1 \nm{\sum_{j,l}C(M,n)a_{j,l}\overline{z_{j,l}}^n r_{j,l}(\tau)\frac{(1-\abs{z_{j,l}})^{M-\frac{1}{p}}\omg(z_{j,l})^{-\frac{1}{q}}}{(1-\overline{z_{j,l}}re^{i\cdot})^{M+n}}}^s_{L^s(\mu_r)} d\tau \right )^{\frac{t}{s}} \\
				&\asymp \left (\int_0^1\int_0^{2\pi} \abs{F_\tau^{(n)}(re^{i\t})}^s d\mu_r(\t)d\tau  \right)^{\frac{t}{s}}.
			\end{split}
		\end{equation}
		
		Then, by \eqref{eq: kinchine 1} and \eqref{eq: kinchine 2}, it follows that
		\begin{equation}
			\begin{split}\label{eq:nec1}
				C^t \nm{\{a_{j,l}\}}_{\ell^{p,q}}^t &\gtrsim \int_0^1 \left (\int_0^1\int_0^{2\pi} \abs{F_\tau^{(n)}(re^{i\t})}^s d\mu_r(\t)d\tau  \right)^{\frac{t}{s}}d\nu(r) \\
				&= \sum_{j=1}^\infty \int_{r_{j-1}}^{r_j}\left (\int_0^1\int_0^{2\pi} \abs{F_\tau^{(n)}(re^{i\t})}^s d\mu_r(\t)d\tau  \right)^{\frac{t}{s}}d\nu(r).
			\end{split}
		\end{equation}
		Next, for every $0\leq r<1$ take $j\in\N$ such that $r_{j-1}\leq r <r_j$. Then,  by  Fubini's theorem and Khinchine's inequality (\cite[Appendix A]{Duren})
		\begin{equation}
			\begin{split}\label{eq:nec2}
				&\int_0^1\int_0^{2\pi} \abs{F_\tau^{(n)}(re^{i\t})}^s d\mu_r(\t)d\tau  \\ &\asymp 
				\int_0^{2\pi} \int_0^1 \abs{\sum_{k,l}C(m,n)a_{k,l}r_{k,l}\overline{z_{k,l}}^n(\tau)\frac{(1-\abs{z_{k,l}})^{M-\frac{1}{p}}\omg(z_{k,l})^{-\frac{1}{q}}}{(1-\overline{z_{k,l}}re^{i\t})^{M+n}}}^s d\tau \, d\mu_r(\t)   \\
				&\asymp \int_0^{2\pi} \left ( \sum_{k,l} \abs{a_{k,l}}^2\frac{(1-\abs{z_{k,l}}^2)^{2M-\frac{2}{p}}\omg(z_{k,l})^{-\frac{2}{q}}}{\abs{1-\overline{z_{k,l}}re^{i\t}}^{2M+2n}} \right )^{\frac{s}{2}} d\mu_r(\t) \\
				&\geq \int_0^{2\pi} \left ( \sum_{l=0}^{K^{j+2}-1} \abs{a_{j,l}}^2\frac{(1-\abs{z_{j,l}}^2)^{2M-\frac{2}{p}}\omg(z_{j,l})^{-\frac{2}{q}}}{\abs{1-\overline{z_{j,l}}re^{i\t}}^{2M+2n}} \right )^{\frac{s}{2}} d\mu_r(\t).
			\end{split}
		\end{equation}
The last step is based on the following inequalities:
\begin{equation}
\begin{split}\label{eq:inequalities}
    &\frac{1-\abs{z_{j,l}}^2}{\abs{1-\overline{z_{j,l}}re^{i\t}}}
    \gtrsim \chi_{I_{K^{j+2},l}}(\t),
    \quad j\in \N,\quad r_{j-1}\le r < r_j,\quad l=0,1,\ldots,K^{j+2}-1,
    \\
    &\omg(z_{j,l})\asymp \omg(r_j)
    \quad \text{and} \quad
    (1-\abs{z_{j,l}}^2)\asymp K^{-j},
    \quad j\in \N,\quad l=0,1,\ldots,K^{j+2}-1.
\end{split}
\end{equation}
Using \eqref{eq:nec1}, \eqref{eq:nec2} and \eqref{eq:inequalities} we get 
		\begin{equation*}
			\begin{split}
			&	C^t\nm{\{a_{j,l}\}}_{\ell^{p,q}}^t 
			\\	&\gtrsim \sum_{j=1}^\infty \int_{r_{j-1}}^{r_j}  \left (\int_0^{2\pi} \left (\sum_{l=0}^{K^{j+2}-1} \abs{a_{j,l}}^2 \chi_{I_{K^{j+2},l}}(\t) K^{j(2n+\frac{2}{p})}\omg(r_j)^{-\frac{2}{q}} \right )^{\frac{s}{2}} d\mu_r(\t)\right )^{\frac{t}{s}}d\nu(r)\\
				&\asymp \sum_{j=1}^{\infty}\int_{r_{j-1}}^{r_j}K^{jt(n+\frac{1}{p})}\omg(r_j)^{-\frac{t}{q}}\left (\sum_{l=0}^{K^{j+2}-1}|a_{j,l}|^s \mu_r(I_{K^{j+2},l})  \right )^{\frac{t}{s}} d\nu(r).
			\end{split}
		\end{equation*}
		Consequently (ii) holds and 
		$$
		\inf\{C>0: \text{ such that } \eqref{eq: carleson-type inequality} \text{ holds} \} \lesssim\inf\{C>0: \text{ such that } \eqref{eq: carleson-type discrete} \text{ holds} \} .
		$$ 
		$\mathbf{(ii) \implies (i)}$.
		Asumme now that (ii) holds. 
Then,  by  Lemma~\ref{lemma: descomposicion norma bloques}

		\begin{equation*}
			\begin{split}
			&	\int_0^1 \left (\int_0^{2\pi}\abs{f^{(n)}(re^{i\t})}^s d\mu_r(\t) \right )^{\frac{t}{s}}d\nu(r)
\\ &= \sum_{j=1}^\infty \int_{r_{j-1}}^{r_j} \left (  \sum_{l=0}^{K^{j+2}-1}\int_{I_{K^{j+2},l}} \abs{f^{(n)}(r e^{i\t})}^s d\mu_r(\t) \right)^{\frac{t}{s}}d\nu(r)
\\ & \lesssim \sum_{j=1}^\infty \int_{r_{j-1}}^{r_j} \left (  \sum_{l=0}^{K^{j+2}-1}\abs{f^{(n)}_{j,l}}^s \mu_r(I_{K^{j+2},l}) \right)^{\frac{t}{s}}d\nu(r)\\
				&= \sum_{j=1}^\infty \int_{r_{j-1}}^{r_j}K^{jt(n+\frac{1}{p})}\omg(r_j)^{-\frac{t}{q}} \left (  \sum_{l=0}^{K^{j+2}-1}\left(\abs{f^{(n)}_{j,l}}K^{-j(n+\frac{1}{p})}\omg(r_j)^{\frac{1}{q}}\right)^s \mu_r(I_{K^{j+2},l}) \right)^{\frac{t}{s}}d\nu(r) \\
				&\leq C^t \nm{\left\{\abs{f^{(n)}_{j,l}}K^{-j(n+\frac{1}{p})}\omg(r_j)^{\frac{1}{q}}\right\}}_{\ell^{p,q}}^t \lesssim C^t \nm{f}_{\Apq{p}{q}}^t,
			\end{split}
		\end{equation*}
where $C$  comes from \eqref{eq: carleson-type discrete}. 
		So (i) holds and 
		$$\inf\{C>0: \text{ such that } \eqref{eq: carleson-type inequality} \text{ holds} \} \gtrsim \inf\{C>0: \text{ such that } \eqref{eq: carleson-type discrete} \text{ holds} \}.$$
This finishes the proof.
	\end{Prf}
\medskip
	\par Prior to dealing with the particular case $d\mu_r(\t)=\abs{G(re^{i\t})}^s\frac{d\theta}{2\pi}$ and  $\nu\in\DDD$, let us recall the following concepts and results. 
	Let $0<p,q,s,t<\infty$, the sequence $\{b_{j,l}\}$ multiplies $\ell^{p,q}$ into $\ell^{s,t}$ if there exists a constant $C>0$ such that
	$$
	\nm{\{a_{j,l} b_{j,l}\}}_{\ell^{s,t}}\leq C \nm{\{a_{j,l}\}}_{\ell^{p,q}},\quad \text{for each}\, \{a_{j,l}\}\in \ell^{p,q}.
	$$
	We denote by $[\ell^{p,q},\ell^{s,t}]$ the space of multipliers from $\ell^{p,q}$ to $\ell^{s,t}$ and 
	$$
	\nm{\{b_{j,l}\}}_{[\ell^{p,q},\ell^{s,t}]} = \inf\set{ C>0: \nm{\{a_{j,l} b_{j,l}\}}_{\ell^{s,t}}\leq C \nm{\{a_{j,l}\}}_{\ell^{p,q}},\quad \{a_{j,l}\}\in \ell^{p,q}}.
	$$
	The following description of $[\ell^{p,q},\ell^{s,t}]$ will be useful for our purpose.
	\begin{letterlemma}\cite[Lemma 7]{luecking}\label{lemma: multiplicadores luecking}
		Let $0<p,q,s,t<\infty$, then $[\ell^{p,q},\ell^{s,t}]=
		\ell^{s\left (\frac{p}{s} \right)',t\left (\frac{q}{t} \right)'}$. Moreover,
		$$
		\nm{\{b_{j,l}\}}_{[\ell^{p,q},\ell^{s,t}]} \asymp \nm{\{b_{j,l}\}}_{\ell^{s\left (\frac{p}{s} \right)',t\left (\frac{q}{t} \right)'}}.
		$$
	\end{letterlemma}
	
	\begin{proposition}\label{propo: caracterización G nu}
		Let $0<p,q,s,t<\infty$, $G\in\H(\D)$, $\om,\nu \in \DDD$ and $K=\max\{ K(\om), K(\nu)\}$ where  $K(\om) \in \N \setminus \{1\}$, $ K(\nu) \in \N \setminus \{1\}$ and both are  such that \eqref{eq: def Dcheck} holds.  Then, the following conditions  are equivalent:
		\begin{itemize}
			\item[(i)] There exists a constant $C>0$ such that 
			\begin{equation}\label{eq: carleson-type G nu}
				\left (\int_0^1 \left (\int_0^{2\pi}\abs{f^{(n)}(re^{i\t})}^s \abs{G(re^{i\t})}^s\frac{d\theta}{2\pi} \right)^{\frac{t}{s}}\nu(r)dr \right)^\frac{1}{t} \leq C \nm{f}_{\Apq{p}{q}};
			\end{equation}
			\item[(ii)] 
$\left \{K^{j(n+\frac{1}{p})}\nug(r_j)^{\frac{1}{t}}\omg(r_j)^{-\frac{1}{q}} (G_{[s]}(r_{j-1}))_{K^{j+2},l} \right \}\in \ell^{s\left (\frac{p}{s}\right )',t\left (\frac{q}{t}\right )'}.$
		\end{itemize}
		Moreover,
        $$
		 \inf\{C>0: \text{ such that } \eqref{eq: carleson-type G nu} \text{ holds}\} 
        \asymp \nm{\left \{K^{j(n+\frac{1}{p})}\nug(r_j)^{\frac{1}{t}}\omg(r_j)^{-\frac{1}{q}} (G_{[s]}(r_{j-1}))_{K^{j+2},l} \right \}}_{\ell^{s\left (\frac{p}{s}\right )',t\left (\frac{q}{t}\right )'}}.
        $$
	\end{proposition}
	\begin{proof}
    In order to prove this equivalence, we will use the proof (and we keep the same notation) of Proposition~\ref{propo: carleson type} in the particular case $d\mu_r(\t)=\abs{G(re^{i\t})}^s\frac{d\theta}{2\pi}$ and $d\nu(r)=\nu(r)dr$. Firstly, assume that $(i)$ holds. 
    Observe that by  Lemma~\ref{lemma: weight integral radius}
		\begin{equation}\label{eq: wgorro equiv w int}
			\int_{r_{j-1}}^{r_j} \nu(r)dr \asymp \nug(r_j),\quad j\in \N.
		\end{equation}
Therefore, by
\eqref{eq:nec1}
		\begin{equation*}
			\begin{split}
				C^t \nm{\{a_{j,l}\}}_{\ell^{p,q}}^t &\gtrsim \int_0^1 \left (\int_0^1\int_0^{2\pi} \abs{F_\tau^{(n)}(re^{i\t})}^s \abs{G(re^{i\t})}^s \frac{d\theta}{2\pi} d\tau  \right)^{\frac{t}{s}}\nu(r)dr \\
				&= \sum_{j=1}^\infty \int_{r_{j-1}}^{r_j}\left (\int_0^1 M_s^s(r,F_\tau^{(n)}\cdot G) d\tau  \right)^{\frac{t}{s}}\nu(r)dr \\
				&\gtrsim \sum_{j=1}^\infty \nug(r_j) \left ( \int_0^1  M_s^s(r_{j-1},F_\tau^{(n)}\cdot G) d\tau  \right )^{\frac{t}{s}},
			\end{split}
		\end{equation*}
		where $C$ comes from \eqref{eq: carleson-type G nu}.
		Next, joining the above inequality with \eqref{eq:nec2} (applied to $r=r_{j-1}$) and \eqref{eq:inequalities}, we get 
		\begin{equation}
			\begin{split}
				& C^t \nm{\{a_{j,l}\}}_{\ell^{p,q}}^t \\ &\gtrsim \sum_{j=1}^\infty \nug(r_j)   \left (\int_0^{2\pi} \left (\sum_{l=0}^{K^{j+2}-1} \abs{a_{j,l}}^2 \chi_{I_{K^{j+2},l}}(\t) K^{j(2n+\frac{2}{p})}\omg(r_j)^{-\frac{2}{q}} \right )^{\frac{s}{2}} \abs{G(r_{j-1}e^{i\t})}^s\frac{d\theta}{2\pi} \right )^{\frac{t}{s}}\\
				&\asymp \sum_{j=1}^{\infty} K^{jt (n+\frac{1}{p} )}\nug(r_j)\omg(r_j)^{-\frac{t}{q}}\left (\sum_{l=0}^{K^{j+2}-1}|a_{j,l}|^s (G_{[s]}(r_{j-1}))_{K^{j+2},l}^s  \right )^{\frac{t}{s}}, \quad \{a_{j,l}\}\in \ell^{p,q},
			\end{split}
		\end{equation}
		it means $\left \{K^{j(n+\frac{1}{p})}\nug(r_j)^{\frac{1}{t}}\omg(r_j)^{-\frac{1}{q}} (G_{[s]}(r_{j-1}))_{K^{j+2},l} \right \}\in[\ell^{p,q},\ell^{s,t}]$. Therefore, (ii) holds and the inequality
		$$
		\nm{\left \{K^{j(n+\frac{1}{p})}\nug(r_j)^{\frac{1}{t}}\omg(r_j)^{-\frac{1}{q}} (G_{[s]}(r_{j-1}))_{K^{j+2},l} \right \}}_{\ell^{s\left (\frac{p}{s}\right )',t\left (\frac{q}{t}\right )'}}\lesssim \inf\{C>0: \text{ such that } \eqref{eq: carleson-type G nu} \text{ holds}\},
		$$
		follows now from Lemma~\ref{lemma: multiplicadores luecking}. 
		
		Asumme now that (ii) holds.  
Arguing as in the proof of (ii)$\Rightarrow$(i) in Proposition~\ref{propo: carleson type}, 
 and using \eqref{eq: wgorro equiv w int},  Lemma~\ref{lemma: multiplicadores luecking} and Lemma~\ref{lemma: descomposicion norma bloques}
		\begin{equation*}
			\begin{split}
    &\int_0^1 \left (\int_0^{2\pi}\abs{f^{(n)}(re^{i\t})}^s \abs{G(re^{i\t})}^s \frac{d\theta}{2\pi}\right )^{\frac{t}{s}}\nu(r)dr \\ 
    &\leq \sum_{j=1}^\infty \nug(r_j) \left (  \sum_{l=0}^{K^{j+3}-1}\int_{I_{{K^{j+3}},l}} \abs{f^{(n)}(r_{j} e^{i\t})}^s \abs{G(r_{j}e^{i\t})}^s \frac{d\theta}{2\pi} \right)^{\frac{t}{s}} \\
    &\lesssim \sum_{j=1}^\infty \nug(r_j)   \left (  \sum_{l=0}^{K^{j+3}-1} \abs{f^{(n)}_{j+1,l}}^s (G_{[s]}(r_j))_{K^{j+3},l}^s \right)^{\frac{t}{s}}\\			&\lesssim \sum_{j=1}^\infty \nug(r_j) \left (  \sum_{l=0}^{K^{j+2}-1}  \abs{f^{(n)}_{j,l}}^s (G_{[s]}(r_{j-1}))_{K^{j+2},l}^s \right )^{\frac{t}{s}} \\
	&= \sum_{j=1}^\infty K^{jt (n+\frac{1}{p} )}\nug(r_j)\omg(r_j)^{-\frac{t}{q}}  \left (  \sum_{l=0}^{K^{j+2}-1} \left (\abs{f^{(n)}_{j,l}}K^{-j(n+\frac{1}{p})}\omg(r_j)^{\frac{1}{q}}\right )^s (G_{[s]}(r_{j-1}))_{K^{j+2},l}^s \right)^{\frac{t}{s}}\\
	&\lesssim \nm{\left \{K^{j(n+\frac{1}{p})}\nug(r_j)^{\frac{1}{t}}\omg(r_j)^{-\frac{1}{q}} (G_{[s]}(r_{j-1}))_{K^{j+2},l} \right \}}_{\ell^{s\left (\frac{p}{s}\right )',t\left (\frac{q}{t}\right )'}}^t \nm{\left\{\abs{f^{(n)}_{j,l}}K^{-j(n+\frac{1}{p})}\omg(r_j)^{\frac{1}{q}}\right\}}_{\ell^{p,q}}^t \\
	&\lesssim \nm{\left \{K^{j(n+\frac{1}{p})}\nug(r_j)^{\frac{1}{t}}\omg(r_j)^{-\frac{1}{q}} (G_{[s]}(r_{j-1}))_{K^{j+2},l} \right \}}_{\ell^{s\left (\frac{p}{s}\right )',t\left (\frac{q}{t}\right )'}}^t \nm{f}_{\Apq{p}{q}}^t,
			\end{split}
		\end{equation*}
		so (ii) holds and 
		$$ \inf\{C>0: \text{ such that } \eqref{eq: carleson-type G nu} \text{ holds}\} \lesssim \nm{\left \{K^{j(n+\frac{1}{p})}\nug(r_j)^{\frac{1}{t}}\omg(r_j)^{-\frac{1}{q}} (G_{[s]}(r_{j-1}))_{K^{j+2},l} \right \}}_{\ell^{s\left (\frac{p}{s}\right )',t\left (\frac{q}{t}\right )'}}.$$
	\end{proof}
Next, the equivalence (i)$\Leftrightarrow$(ii) of Theorem~\ref{th2: charact Tg} is obtained as a byproduct of Proposition~\ref{propo: caracterización G nu}. 

    \begin{corollary}\label{coro: discrete caract Tg}
        Let $0<p,q,s,t<\infty$ , $g\in\H(\D)$, $\om \in \DDD$ and $K=K(\om) \in \N \setminus \{1\}$ such that \eqref{eq: def Dcheck} holds.  Then, the following conditions  are equivalent:
		\begin{itemize}
			\item[(i)] $T_g:\Apq{p}{q}\to \Apq{s}{t}$ is bounded;
            \item[(ii)] 
 $\left \{K^{j(\frac{1}{p}-1)}\omg(r_j)^{\frac{1}{t}-\frac{1}{q}}  \left (g'_{[s]}(r_{j-1}) \right)_{K^{j+2},l} \right \}\in \ell^{s\left (\frac{p}{s}\right )',t\left (\frac{q}{t}\right )'}.$
		\end{itemize}
		Moreover,
		$$
		\nm{T_g}_{\Apq{p}{q}\to\Apq{s}{t}} \asymp \nm{\left \{K^{j(\frac{1}{p}-1)}\omg(r_j)^{\frac{1}{t}-\frac{1}{q}}  \left (g'_{[s]}(r_{j-1}) \right)_{K^{j+2},l} \right \}}_{\ell^{s\left (\frac{p}{s}\right )',t\left (\frac{q}{t}\right )'}}.
		$$
    \end{corollary}
    \begin{proof}
    By \eqref{eq: L-P estimate} (i) is equivalent to \eqref{eq: carleson-type G nu} with $n=0$, $G=g'$ and $\nu(r)=\om_{[t]}(r)=\om(r)(1-r)^t$, and 
    $$
    \nm{T_g}_{\Apq{p}{q}\to\Apq{s}{t}}\asymp \inf\{C>0: \text{ such that } \eqref{eq: carleson-type G nu} \text{ holds}\}.
    $$
   Now, we observe that for $t>0$ and $\om\in\DDD$, $\om_{[t]}(r)=\om(r)(1-r)^t \in \DDD$, see \cite{PelRat,PelRosa}. Moreover, if $K=K(\om), \in \N \setminus \{1\}$  is such that \eqref{eq: def Dcheck} holds for $\omega$,
\eqref{eq: def Dcheck} also  holds for $\om_{[t]}$. Consequently, 
 the result follows from Proposition~\ref{propo: caracterización G nu}.
    \end{proof}
Taking $n=0$, $G=g$ and $\nu=\om$ in the case of $M_g$ and bearing in mind that the boundedness of $S_g:\Apq{p}{q}\to \Apq{s}{t}$ is equivalent to the boundedness of $M_g: A^{p,q}_{\om_{[q]}}\to  A^{s,t}_{\om_{[t]}}$, the proof of Corollary~\ref{coro: discrete caract Tg} can be mimicked to get the a discrete description of the symbols $g$ such that $L_g:\Apq{p}{q}\to \Apq{s}{t}$, $L_g\in \{M_g, S_g\}$, is bounded.
    \begin{corollary}\label{coro: discrete caract Sg Mg}
        Let $0<p,q,s,t<\infty$ , $g\in\H(\D)$, $\om \in \DDD$ and $K=K(\om)\in \N \setminus \{1\}$ such that \eqref{eq: def Dcheck} holds.  Then, the following conditions  are equivalent:
		\begin{itemize}
			\item[(i)] $S_g:\Apq{p}{q}\to \Apq{s}{t}$ is bounded;
            \item[(ii)] $M_g:\Apq{p}{q}\to \Apq{s}{t}$ is bounded;
            \item[(iii)] 
 $\left \{K^{\frac{j}{p}}\omg(r_j)^{\frac{1}{t}-\frac{1}{q}}  \left (g_{[s]}(r_{j-1}) \right)_{K^{j+2},l} \right \}\in \ell^{s\left (\frac{p}{s}\right )',t\left (\frac{q}{t}\right )'}.$
		\end{itemize}
		Moreover,
		$$
		\nm{S_g}_{\Apq{p}{q}\to\Apq{s}{t}} \asymp \nm{M_g}_{\Apq{p}{q}\to\Apq{s}{t}} \asymp \nm{\left \{K^{\frac{j}{p}}\omg(r_j)^{\frac{1}{t}-\frac{1}{q}}  \left (g_{[s]}(r_{j-1}) \right)_{K^{j+2},l} \right \}}_{\ell^{s\left (\frac{p}{s}\right )',t\left (\frac{q}{t}\right )'}}.
		$$
    \end{corollary}
	\section{ Proof of Theorem~\ref{th2: charact Tg}}\label{sec:continuous}
In this section, we provide a detailed proof of  Theorem~\ref{th2: charact Tg}.  Bearing in mind Corollary~\ref{coro: discrete caract Sg Mg}, similar 
 arguments work to prove Theorem~\ref{th2: charact SgMg}, so we omit its proof.

We begin proving the sufficiency of the continuous conditions as an straightforward consequence of Hölder's inequality and  the classical Hardy-Littlewood inequality \eqref{HL}.
	\begin{proposition}\label{propo: suff bound omega en D}
		Let $0<p,q,s,t<\infty$,  $g\in\H(\D)$, $\om\in\DDD$ and denote $\frac{1}{\tilde{p}}=\frac{1}{s}-\frac{1}{p}$ and $\frac{1}{\tilde{q}}=\frac{1}{t}-\frac{1}{q}$. Then, the following are sufficient conditions so that $T_g:\Apq{p}{q}\to \Apq{s}{t}$ is bounded:
		\begin{itemize}
			\item[(a)] If $p\leq s$ and $q\leq t$, 
			$$
			\sup_{z\in\D} \abs{g'(z)}\omg(z)^{\frac{1}{\tilde{q}}}(1-|z|)^{1+\frac{1}{\tilde{p}}}<\infty;
			$$
			\item[(b)] If $s<p$ and $q\leq t$,
			$$
			\sup_{0\le r<1} (1-r)\omg(r)^{\frac{1}{\tilde{q}}}M_{\tilde{p}}(r,g')<\infty;
			$$
			\item[(c)] If $p\leq s$ and $t<q$,
			$$
			|g'(z)|(1-|z|)^{1+\frac{1}{\tilde{p}}}\in L_\om^{\infty, \tilde{q}};
			$$
			\item[(d)] If $s<p$ and $t<q$,
			$$
			g\in\Apq{\tilde{p}}{\tilde{q}}.
			$$
		\end{itemize}  
		Moreover,
		$$
		\nm{T_g}_{\Apq{p}{q}\to \Apq{s}{t}}\lesssim \rho_{p,q,s,t,\om}(g)=\begin{cases}
			\sup_{z\in\D} \abs{g'(z)}\omg(z)^{\frac{1}{\tilde{q}}}(1-\abs{z})^{1+\frac{1}{\tilde{p}}}, & \text{if}\; p\leq s, q\leq t \\ \sup_{0\le r<1} (1-r)\omg(r)^{\frac{1}{\tilde{q}}}M_{\tilde{p}}(r,g'),
			& \text{if}\;s<p, q\leq t \\ \nm{ (1-|\cdot|)^{1+\frac{1}{\tilde{p}}}\;g'}_{L_\om^{\infty, \tilde{q}}},
			& \text{if}\;p\leq s, t<q \\ \nm{g-g(0)}_{\Apq{\tilde{p}}{\tilde{q}}}
			& \text{if}\;s<p,t<q.
		\end{cases} 
		$$
	\end{proposition}	   
	\begin{proof}
		We will use that
 \begin{equation}\label{eq: integral mean pointwise}
		M_p(r,f)\leq \frac{\nm{f}_{\Apq{p}{q}}}{\omg(r)^{\frac{1}{q}}},\quad f\in \Apq{p}{q},\quad 0\leq r <1.
		\end{equation}
		(a) If $p\leq s$ and $q \leq t$, joining \eqref{eq: L-P estimate},  \eqref{eq: integral mean pointwise} and \eqref{HL} and  we obtain that
		\begin{equation*}
			\begin{split}
				\nm{T_g(f)}_{\Apq{s}{t}}^t &\asymp \int_0^1 M_s^t(r,f\cdot g')(1-r)^t\om(r)dr\\ &\leq \rho_{p,q,s,t,\om}(g)^t\int_0^1 M_s^t(r,f) (1-r)^{\frac{t}{p}-\frac{t}{s}}\omg(r)^{\frac{t}{q}-1}\om(r)dr \\
				&\lesssim \rho_{p,q,s,t,\om}(g)^t\int_0^1 M_p^t\left (\frac{1+r}{2},f \right)\omg(r)^{\frac{t}{q}-1}\om(r)dr \\ &\leq \rho_{p,q,s,t,\om}(g)^t\nm{f}_{\Apq{p}{q}}^{t-q} \int_0^1  M_p^q\left (\frac{1+r}{2},f \right)\frac{\omg(r)^{\frac{t}{q}-1}}{\omg\left (\frac{1+r}{2}\right )^{\frac{t}{q}-1}}\om(r)dr\\ &\lesssim \rho_{p,q,s,t,\om}(g)^t\nm{f}_{\Apq{p}{q}}^t,\quad f\in\Apq{p}{q},
			\end{split}
		\end{equation*}
		where the last step follows from the fact that $\om\in\DD$ and an integration by parts. \\
		
		(b) If $s<p$ and $q\leq t$, by \eqref{eq: L-P estimate}, Hölder's inequality, the hypothesis and \eqref{eq: integral mean pointwise} 
		\begin{equation*}
			\begin{split}
				\nm{T_g(f)}_{\Apq{s}{t}}^t &\asymp \int_0^1 M_s^t(r,f\cdot g')(1-r)^t\om(r)dr\\ &\leq \int_0^1 M_p^t(r,f)M_{\tilde{p}}^t(r,g')(1-r)^t\om(r)dr \\ 
				&\leq \rho_{p,q,s,t,\om}(g) \int_0^1 M_p^t(r,f) \omg(r)^{\frac{t}{q}-1}\om(r)dr \\ &\leq \rho_{p,q,s,t,\om}(g)^t\nm{f}_{\Apq{p}{q}}^{t-q}\int_0^1 M_p^q(r,f)\om(r)dr\\ 
				&= \rho_{p,q,s,t,\om}(g)^t\nm{f}_{\Apq{p}{q}}^t,\quad f\in\Apq{p}{q}.
			\end{split}
		\end{equation*}
		
		(c) If $p\leq s$ and $t<q$, by \eqref{eq: L-P estimate},  \eqref{HL}, Hölder's inequality and the hypothesis
		\begin{equation*}
			\begin{split}
				\nm{T_g(f)}_{\Apq{s}{t}}^t &\asymp \int_0^1 M_s^t(r,f\cdot g')(1-r)^t\om(r)dr\\ &\leq \int_0^1 M_s^t(r,f)M_\infty^t(r,g')(1-r)^t\om(r)dr \\
				&\lesssim \int_0^1 M_p^t\left (\frac{1+r}{2},f\right )M_\infty^t(r,g')(1-r)^{t\left (1+\frac{1}{\tilde{p}} \right)}\om(r)dr \\
				&\leq  \rho_{p,q,s,t,\om}(g)^t \left( \int_0^1 M_p^q\left (\frac{1+r}{2},f\right )\om(r)dr\right)^{t/q} \\
				&\lesssim \rho_{p,q,s,t,\om}(g)^t \nm{f}_{\Apq{p}{q}}^t,\quad f\in\Apq{p}{q},
			\end{split}
		\end{equation*}
		where the last step follows from the fact that $\om\in\DD$ and an integration by parts. 
		
		(d) If $s<p$ and $t<q$, by \eqref{eq: L-P estimate} and two applications of Hölder's inequality
		\begin{equation*}
			\begin{split}
				\nm{T_g(f)}_{\Apq{s}{t}}^t &\asymp \int_0^1 M_s^t(r,f\cdot g')(1-r)^t\om(r)dr\\ &\leq \int_0^1 M_p^t(r,f)M_{\tilde{p}}^t(r,g')(1-r)^t\om(r)dr \\
				&\leq \nm{f}_{\Apq{p}{q}}^t \left (\int_0^1 M_{\tilde{p}}^{\tilde{q}}(r,g')(1-r)^{\tilde{q}}\om(r)dr \right)^{\frac{t}{\tilde{q}}} \\
				&\asymp \rho_{p,q,s,t,\om}(g)^t \nm{f}_{\Apq{p}{q}}^t ,\quad f\in\Apq{p}{q}.
			\end{split}
		\end{equation*}
	\end{proof}
	Now we are ready to provide a proof of Theorem~\ref{th2: charact Tg}. It is worth mentioning that in some of the cases (a)-(d),  a proof of the implication (iii)$\Rightarrow$(i) can be obtained using different techniques to those employed below. For instance,
in the
case (a),  a proof of (iii)$\Rightarrow$(i) can be obtained testing on  an appropiate family of analytic functions $\{f_a\}_{a\in \D}$.  

	\begin{Prf}\emph{Theorem~\ref{th2: charact Tg}}
		
		The equivalence between (i) and (ii) follows from Corollary~\ref{coro: discrete caract Tg},  and (iii) implies (i) follows from Proposition~\ref{propo: suff bound omega en D}. It remains to prove that (ii) implies (iii) and 
		\begin{equation}\label{eq: sequence norm p<=s,q<=t}
		 \rho_{p,q,s,t,\om}(g) \lesssim \nm{\left \{K^{j(\frac{1}{p}-1)}\omg(r_j)^{\frac{1}{\tilde{q}}} \left (g'_{[s]}(r_{j-1}) \right)_{K^{j+2},l} \right \}}_{\ell^{s\left (\frac{p}{s}\right )',t\left (\frac{q}{t}\right )'}},
		\end{equation}
for each of the four cases (a), (b), (c) and (d).

We  will deal with each case by separated ways. In all of them we will use that for
$N,M\in\N$ with $N\leq M \leq C N$,  there exist two positive constants $C_1(C,s,q),C_2(C,s,q)$, such that
		\begin{equation}\begin{split}\label{eq:amalgama}
		C_1\nm{\{\left (f_{[s]}(r) \right )_{N,l}\}_{l=0}^{N-1}}_{\ell^q} & \leq \nm{\{\left (f_{[s]}(r) \right )_{M,l}\}_{l=0}^{M-1}}_{\ell^q} 
\\ & \leq C_2\nm{\{\left (f_{[s]}(r) \right )_{N,l}\}_{l=0}^{N-1}}_{\ell^q},\, f\in\H(\D),\, 0\leq r <1.
		\end{split}\end{equation}
		In particular, we will use \eqref{eq:amalgama} for $N=E\left (\frac{1}{r_{j+1}-r_j} \right )$ and  $M=K^{j+4}$ for all $j\in\N$. 
		
	\textbf{(a)}  If $p\leq s$ and $q\leq t$, then 
\begin{equation*}\begin{split}
	&	\nm{\left \{K^{j(\frac{1}{p}-1)}\omg(r_j)^{\frac{1}{\tilde{q}}} \left (g'_{[s]}(r_{j-1}) \right)_{K^{j+2},l} \right \}}_{\ell^{s\left (\frac{p}{s}\right )',t\left (\frac{q}{t}\right )'}}
\\ &= \sup_{j\in\N} K^{j(\frac{1}{p}-1)}\omg(r_j)^{\frac{1}{\tilde{q}}} \nm{\set{  \left (g'_{[s]}(r_{j-1}) \right)_{K^{j+2},l}}_{l=0}^{K^{j+2}-1}}_{\ell^\infty}
	\end{split}\end{equation*}
		and
		$$
		\rho_{p,q,s,t,\om}(g) =\sup_{z\in\D} \abs{g'(z)}\omg(z)^{\frac{1}{\tilde{q}}}(1-\abs{z})^{1+\frac{1}{\tilde{p}}} = \sup_{0\leq r <1}(1-r)^{1+\frac{1}{\tilde{p}}} \omg(r)^{\frac{1}{\tilde{q}}}M_\infty(r,g').
		$$
		For $0\leq r <1$ let choose $j\in \N$ such that $r_{j-1}\leq r < r_j$. Then, by Lemma~\ref{lemma: caract dgorro}
		$$
		(1-r)^{1+\frac{1}{\tilde{p}}}\omg(r)^{\frac{1}{\tilde{q}}}M_\infty(r,g') \lesssim K^{-j(1+\frac{1}{\tilde{p}})}\omg(r_j)^{\frac{1}{\tilde{q}}}M_\infty(r_j,g') = K^{j(\frac{1}{p}-1)}\omg(r_j)^{\frac{1}{\tilde{q}}}K^{-\frac{j}{s}}M_\infty(r_j,g').
		$$
		Now by Theorem~\ref{th1: improvement} and \eqref{eq:amalgama}
		$
		M_\infty(r_j,g') \lesssim K^{\frac{j}{s}} \nm{ \left \{\left (g'_{[s]}(r_{j+1}) \right )_{K^{j+4},l} \right\}_{l=0}^{K^{j+4}-1}}_{\ell^\infty}. 
		$
Therefore, using   Lemma~\ref{lemma: caract dgorro} again
		\begin{equation*}
			\begin{split}
		(1-r)^{1+\frac{1}{\tilde{p}}}\omg(r)^{\frac{1}{\tilde{q}}}M_\infty(r,g') &\lesssim K^{(j+2)(\frac{1}{p}-1)}\omg(r_{j+2})^{\frac{1}{\tilde{q}}}
\nm{ \left \{\left (g'_{[s]}(r_{j+1}) \right )_{K^{j+4},l} \right\}_{l=0}^{K^{j+4}-1}}_{\ell^\infty} \\
		&\leq \nm{\left \{K^{j(\frac{1}{p}-1)}\omg(r_j)^{\frac{1}{\tilde{q}}}  \left (g'_{[s]}(r_{j-1}) \right)_{K^{j+2},l}\right \}}_{\ell^{s\left (\frac{p}{s}\right )',t\left (\frac{q}{t}\right )'}}.
			\end{split}
		\end{equation*}
		The proof of this case concludes by taking supremum on $r$. 
		
		\textbf{(b)} If $s<p$ and $q\leq t$, then 
\begin{equation*}
\begin{split}
& \nm{\left\{
K^{j\left(\frac{1}{p}-1\right)}
\omg(r_j)^{\frac{1}{\tilde{q}}}
\left(g'_{[s]}(r_{j-1})\right)_{K^{j+2},l}
\right\}}_{\ell^{s\left(\frac{p}{s}\right)',\, t\left(\frac{q}{t}\right)'}}
\\
&=
\sup_{j\in\N}
K^{j\left(\frac{1}{p}-1\right)}
\omg(r_j)^{\frac{1}{\tilde{q}}}
\nm{
\left\{
\left(g'_{[s]}(r_{j-1})\right)_{K^{j+2},l}
\right\}_{l=0}^{K^{j+2}-1}
}_{\ell^{\tilde{p}}}
\end{split}
\end{equation*}
		and
		$$
		\rho_{p,q,s,t,\om}(g) =\sup_{0\le r<1} (1-r)\omg(r)^{\frac{1}{\tilde{q}}}M_{\tilde{p}}(r,g').
		$$
		For $0\leq r <1$ let choose $j\in \N$ such that $r_{j-1}\leq r < r_j$,  then by Lemma~\ref{lemma: caract dgorro}
		$$
		(1-r)\omg(r)^{\frac{1}{\tilde{q}}}M_{\tilde{p}}(r,g') \lesssim K^{-j}\omg(r_j)^{\frac{1}{\tilde{q}}}M_{\tilde{p}}(r_j,g') = K^{j ( \frac{1}{p}-1)}\omg(r_j)^{\frac{1}{\tilde{q}}}K^{-\frac{j}{p}}M_{\tilde{p}}(r_j,g').
		$$
		On the other hand, since $\tilde{p}=s\frac{p}{p-s}>s$, it follows from   Theorem~\ref{th1: improvement} and and \eqref{eq:amalgama} that
		$$
		M_{\tilde{p}}(r_j,g') \lesssim 
 K^{\frac{j}{p}} \nm{ \left\{\left (g'_{[s]}(r_{j+1}) \right )_{K^{j+4},l} \right\}_{l=0}^{K^{j+4}-1}}_{\ell^{\tilde{p}}}.
		$$
		Joining these two inequalities, and applying Lemma~\ref{lemma: caract dgorro}
\begin{equation*}
\begin{split}
&(1-r)\omg(r)^{\frac{1}{\tilde{q}}} M_{\tilde{p}}(r,g')
\\ &\lesssim
K^{(j+2)\left(\frac{1}{p}-1\right)}
\omg(r_{j+2})^{\frac{1}{\tilde{q}}}
\nm{
\left\{
\left(g'_{[s]}(r_{j+1})\right)_{K^{j+4},l}
\right\}_{l=0}^{K^{j+4}-1}
}_{\ell^{\tilde{p}}}
\\
&\leq
\nm{
\left\{
K^{j\left(\frac{1}{p}-1\right)}
\omg(r_j)^{\frac{1}{t}-\frac{1}{q}}
\left(
(g'_{[s]}(r_{j-1}))
\right)_{K^{j+2},l}
\right\}
}_{\ell^{s\left(\frac{p}{s}\right)',\, t\left(\frac{q}{t}\right)'}} .
\end{split}
\end{equation*}
		and we conclude the proof by taking supremum on $r$.
		
		\textbf{(c)} If $p\leq s$ and $t<q$, then 
		\begin{equation*}\begin{split}
 &\nm{\left\{K^{j\left(\frac{1}{p}-1\right)}\omg(r_j)^{\frac{1}{\tilde{q}}}\left((g'_{[s]}(r_{j-1}))\right)_{K^{j+2},l}\right\} }_{\ell^{s\left(\frac{p}{s}\right)',\, t\left(\frac{q}{t}\right)'}}
\\ & =
\left(
\sum_{j=1}^\infty
K^{j\left(\frac{1}{p}-1\right)\tilde{q}}
\omg(r_j)
\nm{
\left\{
\left(
(g'_{[s]}(r_{j-1}))
\right)_{K^{j+2},l}
\right\}_{l=0}^{K^{j+2}-1}
}_{\ell^\infty}^{\tilde{q}}
\right)^{\frac{1}{\tilde{q}}}
\end{split}\end{equation*}
		and
		$$
		\rho_{p,q,s,t,\om}(g) =\nm{ (1-|\cdot|)^{1+\frac{1}{\tilde{p}}}\;g'}_{L_\om^{\infty, \tilde{q}}} = \left (\int_0^1 M_\infty^{\tilde{q}}(r,g') (1-r)^{\left (1+\frac{1}{\tilde{p}}\right )\tilde{q} }\om(r)dr\right )^{\frac{1}{\tilde{q}}}.
		$$
		Bearing in mind  Lemma~\ref{lemma: weight integral radius} 
	\begin{equation*}
			\begin{split}
				\rho_{p,q,s,t,\om}(g)^{\tilde{q}} &=\sum_{j=1}^\infty \int_{r_{j-1}}^{r_{j}}M_\infty^{\tilde{q}}(r,g') (1-r)^{\left (1+\frac{1}{\tilde{p}} \right)\tilde{q} } \om(r)dr \\
				&\lesssim \sum_{j=1}^\infty K^{j\left (\frac{1}{p}-1 \right)\tilde{q}}\omg(r_j) K^{-j\frac{\tilde{q}}{s}}M_\infty^{\tilde{q}}(r_{j},g').
			\end{split}
		\end{equation*}

		Now, arguing as we did in the case (a),  applying Theorem~\ref{th1: improvement},  \eqref{eq:amalgama} and  Lemma~\ref{lemma: caract dgorro}
		\begin{equation*}
			\begin{split}				\rho_{p,q,s,t,\om}(g)^{\tilde{q}}&\lesssim  \sum_{j=1}^\infty K^{j\left (\frac{1}{p}-1 \right)\tilde{q}}\omg(r_j) 
  \nm{\left\{\left((g'_{[s]}(r_{j+1}))\right)_{K^{j+4},l}\right\}_{l=0}^{K^{j+4}-1}}_{\ell^\infty}^{\tilde{q}}
				\\
				&\lesssim \sum_{j=3}^\infty K^{j\left ( \frac{1}{p}-1\right )\tilde{q}}\omg(r_j) 
\nm{
\left\{
\left(
(g'_{[s]}(r_{j-1}))
\right)_{K^{j+2},l}
\right\}_{l=0}^{K^{j+2}-1}
}_{\ell^\infty}^{\tilde{q}}
				\\ &\leq \nm{\left \{K^{j(\frac{1}{p}-1)}\omg(r_j)^{\frac{1}{t}-\frac{1}{q}} \left(
(g'_{[s]}(r_{j-1}))
\right)_{K^{j+2},l} \right \}}_{\ell^{s\left (\frac{p}{s}\right )',t\left (\frac{q}{t}\right )'}}^{\tilde{q}},
			\end{split}
		\end{equation*}
		concluding the proof of this case.
		
		\textbf{(d)} If $s<p$ and $t<q$, then 
\begin{equation*}\begin{split} 
&\nm{\left\{K^{j\left(\frac{1}{p}-1\right)}\omg(r_j)^{\frac{1}{\tilde{q}}}\left((g'_{[s]}(r_{j-1}))\right)_{K^{j+2},l}\right\}}_{\ell^{s\left(\frac{p}{s}\right)',\, t\left(\frac{q}{t}\right)'}}
\\ &
=
\left(\sum_{j=1}^\infty K^{j\left(\frac{1}{p}-1\right)\tilde{q}}\omg(r_j)\nm{\left\{\left((g'_{[s]}(r_{j-1})\right)_{K^{j+2},l}\right\}_{l=0}^{K^{j+2}-1}}_{\ell^{\tilde{p}}}^{\tilde{q}}\right)^{\frac{1}{\tilde{q}}}
\end{split}\end{equation*}
		Now  by \eqref{eq: L-P estimate} and  Lemma~\ref{lemma: weight integral radius}
		\begin{equation*}
			\begin{split}
				\rho_{p,q,s,t,\om}(g) & =\nm{g-g(0)}_{\Apq{\tilde{p}}{\tilde{q}}}^{\tilde{q}} \asymp
 \int_0^1 M_{\tilde{p}}^{\tilde{q}}(r,g')(1-r)^{\tilde{q}}\om(r)dr
\\ & \asymp
 \sum_{j=1}^\infty  \int_{r_{j-1}}^{r_j} M_{\tilde{p}}^{\tilde{q}}(r,g')(1-r)^{\tilde{q}}\om(r)dr
 \lesssim \sum_{j=1}^\infty K^{-j\tilde{q}} \omg(r_j) M_{\tilde{p}}^{\tilde{q}}(r_j,g') \\
				&= \sum_{j=1}^\infty K^{j(\frac{1}{p}-1)\tilde{q}} \omg(r_j) K^{-j\frac{\tilde{q}}{p}}M_{\tilde{p}}^{\tilde{q}}(r_j,g'). 
			\end{split}
		\end{equation*}
		Applying now Theorem~\ref{th1: improvement} and  \eqref{eq:amalgama}, as we did in case (b),  and Lemma~\ref{lemma: caract dgorro}
		\begin{equation*}
			\begin{split}
				\rho_{p,q,s,t,\om}(g)^{\tilde{q}} &\lesssim\sum_{j=1}^\infty K^{j(\frac{1}{p}-1)\tilde{q}} \omg(r_j)   \nm{\left\{ \left((g'_{[s]}(r_{j+1}))\right)_{K^{j+4},l} \right\}_{l=0}^{K^{j+4}-1}}_{\ell^{\tilde{p}}}^{\tilde{q}}\\
				&\asymp \sum_{j=3}^\infty K^{j(\frac{1}{p}-1)\tilde{q}} \omg(r_j) \nm{\left\{\left((g'_{[s]}(r_{j-1}))\right)_{K^{j+2},l}\right\}_{l=0}^{K^{j+2}-1}}_{\ell^{\tilde{p}}}^{\tilde{q}}
 \\
				&\leq \nm{\left \{K^{j(\frac{1}{p}-1)}\omg(r_j)^{\frac{1}{t}-\frac{1}{q}} \left((g'_{[s]}(r_{j-1}))\right)_{K^{j+2},l}
 \right \}}_{\ell^{s\left (\frac{p}{s}\right )',t\left (\frac{q}{t}\right )'}}^{\tilde{q}}.
			\end{split}
		\end{equation*}
This finishes the proof.

	\end{Prf} 
	
\end{document}